\documentclass[a4paper]{article}
\usepackage[a4paper,top=3cm,bottom=2cm,left=3cm,right=3cm,marginparwidth=1.75cm]{geometry}
\usepackage{hyperref}
\usepackage{amsopn}
\hypersetup{
    colorlinks = true,
    allcolors = {blue},
}
\usepackage{graphicx,amsmath,amsthm,amsfonts,amssymb,epstopdf,bm,bigints}
\usepackage{subcaption,aliascnt,todonotes}
\usepackage{booktabs,array,ragged2e}
\usepackage{color,enumitem,multirow}
\usepackage[nameinlink,capitalize]{cleveref}
\usepackage{adjustbox}
\usepackage{pdflscape}
    {\pagebreak[4]\global\pdfpageattr\expandafter{\the\pdfpageattr/Rotate 90}}%
    {\pagebreak[4]\global\pdfpageattr\expandafter{\the\pdfpageattr/Rotate 0}}%
    
\newcommand{\SeCref}[1]{Section~\ref{#1}}

\crefformat{equation}{\textup{#2(#1)#3}}
\crefrangeformat{equation}{\textup{#3(#1)#4--#5(#2)#6}}
\crefmultiformat{equation}{\textup{#2(#1)#3}}{ and \textup{#2(#1)#3}}
{, \textup{#2(#1)#3}}{, and \textup{#2(#1)#3}}
\crefrangemultiformat{equation}{\textup{#3(#1)#4--#5(#2)#6}}%
{ and \textup{#3(#1)#4--#5(#2)#6}}{, \textup{#3(#1)#4--#5(#2)#6}}{, and \textup{#3(#1)#4--#5(#2)#6}}

\Crefformat{equation}{#2Equation~\textup{(#1)}#3}
\Crefrangeformat{equation}{Equations~\textup{#3(#1)#4--#5(#2)#6}}
\Crefmultiformat{equation}{Equations~\textup{#2(#1)#3}}{ and \textup{#2(#1)#3}}
{, \textup{#2(#1)#3}}{, and \textup{#2(#1)#3}}
\Crefrangemultiformat{equation}{Equations~\textup{#3(#1)#4--#5(#2)#6}}%
{ and \textup{#3(#1)#4--#5(#2)#6}}{, \textup{#3(#1)#4--#5(#2)#6}}{, and \textup{#3(#1)#4--#5(#2)#6}}

\crefdefaultlabelformat{#2\textup{#1}#3}

\newcommand{\vecx}{\mathbf{x}}
\newcommand{\vecy}{\mathbf{y}}

\newcommand{\calK}{\mathcal{K}}
\newcommand{\vecalpha}{{\bm \alpha}}

\newcommand{\vecrho}{{\bm \rho}}
\newcommand{\mrank}{{matrix rank}}
\newcommand{\frank}{{function rank}}

\newcommand{\etal}{\textit{et al}.}

\newcommand{\tablistcommand}{%
    \leavevmode\par\vspace{-\baselineskip}%
}

\newlist{tabitemize}{itemize}{1}
\setlist[tabitemize]{%
    leftmargin = *               ,
  label      = \textbullet     ,
  nosep                        ,
  before     = \tablistcommand ,
  after      = \tablistcommand
}

\newtheorem{theorem}{Theorem}[section]
\newtheorem{lemma}[theorem]{Lemma}

\newtheorem{corollary}[theorem]{Corollary}

\numberwithin{theorem}{section}

\newtheorem{assump}{}
\newenvironment{myassump}[2][]
  {\renewcommand \theassump{#2} \begin{assump}}
  {\end{assump}}

\newcommand{\norm}[1]{\left\lVert #1 \right\rVert}
\newcommand{\abs}[1]{\left\lvert #1 \right\rvert}
\renewcommand{\Re}{\operatorname{Re}}
\renewcommand{\Im}{\operatorname{Im}}
\newcommand{\real}[1]{\Re\left( #1 \right)}
\newcommand{\imag}[1]{\Im\left( #1 \right)}
\newcommand{\dif}{\,\mathrm{d}}

\title{On the numerical rank of radial basis function kernels in high dimensions}
\author{
  Ruoxi Wang
  \thanks{
    Institute for Computational and Mathematical Engineering,
    Stanford University,
    Email: ruoxi@stanford.edu
  }
\and
 Yingzhou Li
  \thanks{
     Department of Mathematics,
    Duke University,
    Email: yingzhou.li@duke.edu
  }
\and
 Eric Darve
 \thanks{
    Department of Mechanical Engineering,
    Stanford University,
    Email: darve@stanford.edu
  }
}
\date{}

\begin{document}

\maketitle

\begin{abstract}
    Low-rank approximations are popular methods to reduce the high computational cost of algorithms involving large-scale kernel matrices.
    The success of low-rank methods hinges on the matrix rank of the kernel matrix, and in practice, these methods are effective even for high-dimensional datasets. Their practical success motivates our analysis of the \emph{\frank{}},  an upper bound of the \mrank{}. In this paper, we consider radial basis functions (RBF),  approximate the
    {RBF} kernel with a low-rank representation that is a finite sum of separate products and provide explicit upper bounds on the \frank{} and the $L_\infty$ error for such approximations. Our three main results are as follows. First, for a fixed precision, the \frank{} of RBFs, in the worst case, grows polynomially with the data dimension. Second, precise error bounds for the low-rank approximations in the $L_\infty$ norm are derived in terms of the function smoothness and the domain diameters. And last, a group pattern in the magnitude of singular values for RBF kernel matrices is observed and analyzed, and is explained by a grouping of the expansion terms in the kernel's low-rank representation. Empirical results verify the theoretical results.
\end{abstract}


\section{Introduction}
\label{sec:intro}
Kernel matrices \cite{Vapnik1998, Cristianini2000, Cortes1995} are widely used across fields including machine learning, inverse problems, graph theory and PDEs~\cite{Wang2015,Rasmussen2005, Hofmann2008, Li2014, Kitanidis2015, Scholkopf2002}. The ability to generate data at the scale of millions and even billions has increased rapidly, posing computational challenges to systems involving large-scale matrices. The lack of scalability has made algorithms that accelerate matrix
computations particularly important. 

There have been algebraic
algorithms proposed to reduce the computational burden, mostly
based on low-rank approximations of the matrix or certain
submatrices~\cite{Wang2015}. The singular value decomposition
(SVD)~\cite{Golub2013} is optimal but has an undesirable cubic
complexity. Many methods \cite{Mahoney2011, Halko2011, Liberty2007, Sarlos2006,
Drineas2005,Gittens2016,Drineas2012, Zhang2010} have been proposed to accelerate the low-rank constructions with an acceptable loss of accuracy. The success of these
low-rank algorithms hinges on a large spectrum gap or a fast decay
of the spectrum of the matrix itself or its submatrices. However,
to our knowledge, there is no theoretical guarantee that these conditions always hold.

Nonetheless, algebraic low-rank techniques
are effective in many cases where the data dimension ranges from
moderate to high, motivating us to study the growth rate of \mrank{}s in high dimensions. A precise analysis of the \mrank{} is nontrivial, and we turn to analyzing its upper bound, that is, the \frank{} of
kernels that will be defined in what follows. The \emph{\frank{}} is the
 number of terms in the minimal separable form of $\mathcal{K}(\vecx, \vecy)$, when $\mathcal{K}$ is approximated by a
finite sum of separate products $h_i(\vecx)g_i(\vecy)$ where $h_i$
and $g_i$ are real-valued functions. If
the \frank{} does not grow exponentially with the data dimension,
neither will the \mrank{}.

If, however, we expand the multivariable function $\mathcal{K}(\vecx, \vecy)$ by expanding each variable in turn using $r$ function basis (per dimension), \emph{i.e.}, 
$$\sum_{i_1, \ldots, i_d, j_1, \ldots, j_d = 1}^r g_1^{(i_1)}(x_1)g_2^{(i_2)}(x_2)\ldots g_d^{(i_d)}(x_d) \;
    h_1^{(j_1)}(y_1)h_2^{(j_2)}(y_2)\ldots h_d^{(j_d)}(y_d)$$ 
where $g_k^{(i_k)}$ and $h_k^{(i_k)}$, respectively, are the $i_k$-th function bases in dimension $k$ for $\vecx$ and $\vecy$. Then, the number of terms will be $r^{2d}$, which grows exponentially with the data dimension. The exponential growth is striking in the sense that even for a moderate dimension, a reasonable accuracy would be difficult to achieve. However, in practice, people have observed much lower \mrank{}s. A plausible reason is that both the functions and the data of practical interest enjoy some special properties, which should be considered when carrying out the analysis. 

The aim of this paper, therefore, is to analytically describe the relationship between the \frank{} and the properties of the function and the data, including measures of function smoothness, the data dimension, and the domain diameter. Such relation has not been described before. We hope the conclusions of this paper on functions can provide some theoretical foundations for the practical success of low-rank matrix algorithms.

In this paper, we present three main results. First, we show that under
common smoothness assumptions and up to some precision, the \frank{}
of RBF kernels grows polynomially with increasing dimension $d$ in the
worst case. Second, we provide explicit $L_\infty$ error bounds
for the low-rank approximations of {RBF} kernel functions. And last,
we explain the observed ``decay-plateau'' behavior of the singular values of smooth RBF kernel matrices.

\subsection{Related Work}

There has been extensive interest in kernel properties in a
high-dimensional setting.

One line of research focuses on the spectrum of kernel matrices. There
is a rich literature on the smallest eigenvalues, mainly concerning 
matrix conditioning. Several papers~\cite{ball1992sensitivity,Narcowich1992,
Schaback1994,Schaback1995} provided lower bounds for the
smallest (in magnitude) eigenvalues. Some work further studied the eigenvalue
distributions. Karoui~\cite{ElKaroui2010} obtained the spectral
distribution in the limit by applying a second-order Taylor expansion
to the kernel function. In particular, Karoui considered kernel
matrices with the $(i,j)$-th entry $K(\vecx_i^T\vecx_j/h^2)$ and
$K(\|\vecx_i-\vecx_j\|_2^2/h^2)$, and showed that as data dimension
$d \rightarrow \infty$, the spectral property is the same as that of
the covariance matrix $\frac{1}{d}XX^T$. Wathen~\cite{Wathen2015}
described the eigenvalue distribution of RBF kernel matrices
more explicitly. Specifically, the authors provided formulas to
calculate the number of eigenvalues that decay like $(1/h)^{2k}$
as $h \rightarrow \infty$, for a given $k$. This group pattern in
eigenvalues was observed earlier in \cite{Fornberg2007} but with
no explanation. The same pattern also occurs in the coefficients
of the orthogonal expansion in the RBF-QR method proposed in
\cite{Fornberg2011}. There have also been studies focusing
on the ``flat-limit'' situation where $h \rightarrow \infty$
\cite{driscoll2002interpolation,schaback2008limit,fornberg2004some}.

Another line of research is on developing efficient methods
for function expansion and interpolation. The goal is to
diminish the exponential dependence on the data dimension
introduced by a tensor-product based approach. Barthelmann
\etal{}~\cite{Barthelmann2000} considered polynomial interpolation
on a sparse grid \cite{Gerstner1998}. Sparse grids are based
on a high-dimensional multiscale basis and involve only $O(N
(\log N)^{d-1})$ degrees of freedom, where $N$ is the number of
grid points in one coordinate direction at the boundary. This is in
contrast with the $O(N^d)$ degrees of freedom from tensor-product
grids. Barthelmann showed that when $d \rightarrow \infty$, the number
of selected points grows as $O(d^k)$, where $k$ is related to the
function smoothness.

Trefethen~\cite{Trefethen2016} commented that to ensure a uniform resolution in all directions, the Euclidean degree of a polynomial (defined as $\|\vecalpha\|_2$ for a multi-index $\vecalpha$) may be the most useful. He investigated the complexity of polynomials with degrees defined by 1-, 2- and $\infty-$ norms and concluded that by using the 2-norm we achieve similar accuracy as with the $\infty$-norm, but with $d\,!$ fewer points.


\subsection{Main Results}

In this paper, we study radial basis functions (RBF). RBFs are functions whose value depends only on the distance to the origin. In our manuscript, for convenience, we will consider the form $\calK(\vecx, \vecy) = f(\norm{\vecx-\vecy}_2^2)$. The square inside the argument of $f$ is to ensure that if $f$ is smooth then the RBF function is smooth as well. If we instead use $\calK(\vecx, \vecy) = f(\norm{\vecx-\vecy}_2)$ and pick $f(u) = u$ for example, the kernel $K$ is not differentiable when $\vecx=\vecy$.

We define the numerical \emph{\frank{}} of a kernel $\calK(\vecx,
\vecy)$ related to error $\epsilon$, to which we will frequently refer. 
\[
    R_\epsilon = \min \Big\{
        r \mid \exists~\{h_i\}_{i=1}^r, \{g_i\}_{i=1}^r \text{, s.t.\
        } \forall~\vecx, \vecy \in \mathbb{R}^{d}, ~\Big|\calK(\vecx,
        \vecy ) - \sum_{i=1}^r g_i(\vecx)h_i(\vecy) \Big| \le \epsilon
\Big\}
\]
where $h_i$ and $g_i$ are real functions on $\mathbb{R}^d$, and the
separable form $\sum_{i=1}^r g_i(\vecx)h_i(\vecy)$ will be referred to as a
\emph{low-rank kernel}, or a \emph{low-rank representation} of rank at most $r$. Note that the rank definition
concerns the \frank{} instead of the \mrank{}.

Our two main results are as follows. First, we show that under
common smoothness assumptions of RBFs and for a fixed precision, the \frank{} for RBF kernels
is a polynomial function of the data dimension $d$. Specifically,
the \frank{} $R = O(d^q)$ where $q$ is related to the low-rank
approximation error. Furthermore, precise and detailed error bounds will be proved.

\begin{figure}[htpb]
    \centering
    \includegraphics[width=0.5\textwidth]{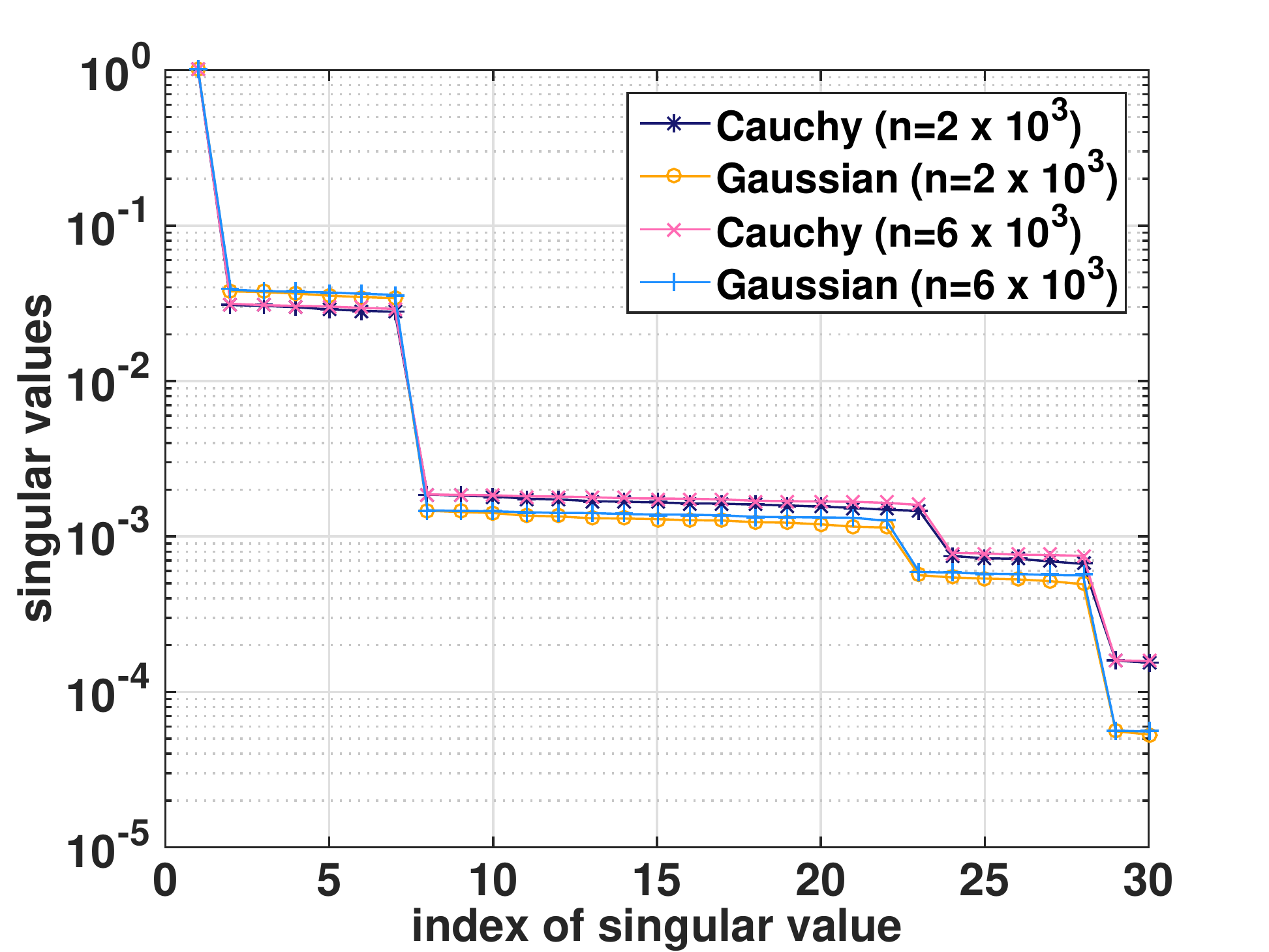}
    \caption{Group patterns in singular values. The singular values are
    normalized and ordered s.t.\ $1 = \sigma_1 \ge \sigma_2 \ge \cdots
    \ge \sigma_n$. The data were randomly generated in dimension 6
    with default random seed in MATLAB. The legend shows the data
    size and the kernel functions: Cauchy ($1/(1 + \norm{x-y}_2^2)$)
    and Gaussian ($\exp(-\norm{x-y}_2^2)$).} \label{fig:plateau_decay}
\end{figure}

Second, we observe that the singular values of RBF kernel
matrices form groups with plateaus. A pictorial example is in
\Cref{fig:plateau_decay}. There are 5 groups (plateau) of singular
values with a sharp drop in magnitude between groups; the group
cardinalities are dependent on the data dimension, but independent of
the data size. We explain this phenomenon by applying an appropriate
analytic expansion of the function and grouping expansion terms
appropriately.

\subsection{Organization}

This paper is organized as follows. \Cref{sec:theorem} presents
our theorems concerning the \frank{} of the approximation of
the {RBF} kernel function, and the $L_\infty$ error bound
of the approximations. \Cref{sec:thm_proofs} provides the
theorem proofs. \Cref{sec:optimal} shows that for a fixed
precision, the polynomial growth rate of the derived rank cannot
be improved. \Cref{sec:experiments} verifies our theorems
experimentally. Finally, in \Cref{sec:plateau_decay}, we investigate
and discuss the group pattern in the singular values of RBF kernel
matrices.


\section{Main theorems}
\label{sec:theorem}

In this section, we present theorems concerning the \frank{} and function and data properties. Each theorem approximates the {RBF} kernels in the $L_\infty$ norm with low-rank kernels where the \frank{} and the error bound are given in explicit formulas. We briefly describe the theorems and then delve into further details. 

The first four theorems consider kernels with two types of smoothness assumptions, and for each type, we present the deterministic result and the probabilistic result in two theorems, respectively. The probabilistic results take into account the concentration of measure for large data dimensions. The separable form is obtained by applying a Chebyshev expansion of $f(z)$ followed by a further expansion of $z = \|\vecx-\vecy\|_2^2$. 

The key advantage of
this approach is that the accuracy of the expansion only depends on $\|\vecx-\vecy\|_2^2$ instead of $(\vecx,\vecy)$, which lies
in a $d$-dimensional space. Assume we have expanded $f(z)$ to order $n$
with error $\epsilon$. Then, we substitute $z = \|\vecx-\vecy\|_2^2$,
 expand the result, and re-arrange the terms to identify the number
of distinct separate products of the form $h(\vecx)g(\vecy)$ in the
final representation. This number becomes our upper bound on the \frank.

The theorems show that for a fixed precision, the
\frank{} grows polynomially with data dimension $d$, and that
the $L_\infty$ error for low-rank approximations decreases with
decreasing diameter of the domain that contains $\vecx$ and
$\vecy$. 

The last theorem considers kernels with finite smoothness assumptions. The separable form is obtained by applying a Fourier expansion of $f(z)$ followed by a Taylor expansion on each Fourier term. Additional to what the previous theorems suggest, the formulas for the error and the \frank{}
capture subtler relations between different parameters, and the theorem shows that the error decreases when either the diameter of the domain that
contains $\vecx$ or that contains $\vecy$ decreases.  Before presenting our theorems, we introduce some notations.

{\bf Notations.} Let ${\bf{E}}(\cdot)$ and ${\bf Var}(\cdot)$ denote
the expectation and variance, respectively. Let
\[ E_{\rho^2} =:
\Bigl\{z = \frac{\rho^2 e^{\imath \theta} + \rho^{-2}e^{-\imath \theta}
}{2} \,\big|\, \theta \in [0, 2\pi)\Bigr\}
\]
be the \emph{Bernstein ellipse} defined on $[-1, 1]$ with parameter
$\rho^2$, an open region bounded by an ellipse. For an arbitrary
interval, the ellipse is scaled and shifted and is referred to as the transformed Bernstein ellipse. For instance, given an
interval $[a,b]$, let $\phi(x)$ be a linear mapping from $[a,b]$
to $[-1,1]$. And the \emph{transformed Bernstein ellipse} for
$[a,b]$ is defined to be $\phi^{-1}(E_{\rho^2})$. In this case, the
parameter $\rho^2$ still characterizes the shape of the transformed
Bernstein ellipse. Therefore, throughout this paper, when we say
a transformed Bernstein ellipse with parameter $\rho^2$, we refer
to the parameter of the Bernstein ellipse defined on [-1, 1]. Let
the function domain be $\Omega_\vecx \times \Omega_\vecy \subset
\mathbb{R}^d \times \mathbb{R}^d$, and we refer to $\Omega_\vecx$
as the \emph{target domain} and $\Omega_\vecy$ as the \emph{source
domain}. We assume the domain is not a manifold, where lower ranks
can be expected. Let the sub-domain containing the data of interest be
$\widetilde{\Omega}_\vecx \times \widetilde{\Omega}_\vecy \subset
\Omega_\vecx \times \Omega_\vecy$.

The following theorems assume the bandwidth parameter $h$ in
$\calK_h(\vecx, \vecy) = f(\|\vecx - \vecy\|_2^2/h^2)$ to be fixed at
1. A scaled kernel $\calK_h(\vecx, \vecy)$ will not be considered because
it can be handled by rescaling the data points instead. We start with some assumptions on the kernel type, function domain, and probabilistic distribution that will be used in the theorems, and then we present our theorems. 

\begin{myassump}{RBF Kernel Assumption}
\label{assumption:function_and_domain}
 Consider a function $f$ and kernel function $\calK(\vecx, \vecy) = 
    f(\norm{\vecx-\vecy}_2^2)$ with $\vecx=(x_1,\ldots,x_d)$ and
    $\vecy=(y_1,\ldots,y_d)$. We assume that $x_i \in [0,D/\sqrt{d}]$,
    $y_i \in [0,D/\sqrt{d}]$, where $D$ is a constant independent of
    $d$. And this implies $\|\vecx-\vecy\|_2^2\le D^2$.
\end{myassump}

\begin{myassump}{Analytic Assumption}
\label{assumption:analytic}
$f$ is analytic in $[0, D^2]$, and is
    analytically continuable to a transformed Bernstein ellipse with
    parameter $\rho_D^2 > 1$, and $|f(x)| \le C_D$ inside the ellipse.
\end{myassump}

\begin{myassump}{Finite Smoothness Assumption}
\label{assumption:finite_smooth}
$f$ and its derivatives through $f^{(q-1)}$ are absolutely continuous on $[0, D^2]$ and the $q$-th derivative has bounded total variation on $[0, D^2]$,
    $
    V\left(\frac{\dif^q f}{\dif x^q}\right)\leq V_q.
    $
\end{myassump}

\begin{myassump}{Probability Distribution Assumption}
\label{assumption:data_distribution}
$x_i$ and $y_i$ are i.i.d.\ random variables, with $x_i \sqrt{d} \in [0,D]$ and $y_i\sqrt{d} \in [0,D]$, and their second moments exist.  Let
  $$E_d = \Big(\sum_{i=1}^d {\bf E}[(x_i-y_i)^2]\Big)^{1/2} = \Big(2{\bf{E}}[(x_i\sqrt{d})^2] - 2\bigl({\bf{E}}[x_i\sqrt{d}]\bigr)^2\Big)^{1/2}$$
  and
  $$\sigma_d^2 = \sum_{i = 1}^d {\bf Var}[(x_i - y_i)^2]$$
  Then, $E_d \in \Theta(1)$ with respect to $d$, \emph{i.e.}, the mean distance between pairs of points neither goes to 0 nor $\infty$ with $d$. And $\sigma_d^2 \in \Theta(\frac{1}{d})$ (a concentration of measure).
\end{myassump}

	\begin{theorem} \label{thm:cheb1}
    Suppose the \Cref{assumption:function_and_domain} and the \Cref{assumption:analytic} hold. 
    Then, for $n \ge 0$, the kernel $\calK$ can be approximated
    in the $L_\infty$ norm by a low-rank kernel $\widetilde{\calK}$
    of \frank{} at most $R(n, d) = \binom{n+d+2}{d+2}$
        \begin{equation}
            \label{eq:low-rank-representation}
            \calK(\vecx,\vecy) = \sum_{i=1}^{R} g_i(\vecx)h_i(\vecy)
            + \epsilon_n = \widetilde{\calK}(\vecx,\vecy) + \epsilon_n
        \end{equation}
     where $\{g_i\}_{i=1}^R$ and $\{h_i\}_{i=1}^R$ are two sequences
     of $d$-variable polynomials. And the error term $\epsilon_n =
     \epsilon_n(D)$ is bounded as
        \begin{equation}
        \label{eq:error_bound_analytic}
        |\epsilon_n(D)| \le \frac{2C_D \rho_D^{-2n}}{\rho_D^2 - 1}
        \end{equation}   
\end{theorem}

   {\bf Remark.}
    If an approximation with a given maximal rank $r$ is requested, we
    need to select an $n(r,d)$ such that $\binom{n(r,d)+d+2}{d+2} \le
    r$. Then, we obtain an approximation with error $|\epsilon_n(D)|
    \le \frac{2C_D \rho_D^{-2n}}{\rho_D^2 - 1}$ and \frank{}
    at most $\binom{n(r,d)+d+2}{d+2} \le r$. The low-rank
    kernel $\widetilde{\calK}$ is of order $2n$, which can be
    revealed from the explicit form of $\widetilde{\calK}$ in
    the proof (see \SeCref{pf:Corollary2_2}). For the space of
    $d$-variate polynomials with maximum total degree $2n$, the
    dimension is $\binom{2n+d}{d}$. In contrast, our upper bound
    is $\binom{n+d+2}{d+2}$. When $d \ge 4$, our formula becomes
    favorable for a large range of $k$.
    
    \begin{corollary} \label{thm:cheb1_corollary}
        Under the same assumptions in $\Cref{thm:cheb1}$ and with
        $n$ fixed, the low-rank kernel approximation, for a fixed
        precision $\epsilon$, is achievable with a rank proportional
        to $d^{\frac{-\log c_1\epsilon}{c_2}}$, where $c_1$ and $c_2$
        are positive constants.
    \end{corollary}
    The proofs of \Cref{thm:cheb1} and \Cref{thm:cheb1_corollary} can
    be found in \SeCref{pf:theorem2_1} and \SeCref{pf:Corollary2_2},
    respectively.

    \Cref{thm:cheb1} suggests that for some precision $\epsilon$, the
    \frank{} grows polynomially with increasing data dimension $d$,
    \emph{i.e.}, $R = O(d^n)$, where $n$ is determined by the desired
    precision $\epsilon$, $D, \rho_D$ and $C_D$. This can be seen from
    $R = \binom{n+d+2}{d+2}$ with $n$ fixed and $d \rightarrow \infty$.

    For a fixed $n$ and for a sub-domain $\widetilde{\Omega}_\vecx
    \cup \widetilde{\Omega}_\vecy$ with diameter $\tilde{D} < D$, the
    error bound decreases, namely in the following sense. In this case, the same function $f$ on the
    sub-domain can be analytically extended to a Bernstein ellipse
    whose parameter is larger than $\rho_D^2$, reducing the error
    bound in \cref{eq:error_bound_analytic}. Therefore, when the diameter of the domain that contains
    our data decreases, we will observe a lower approximation error
    for low-rank approximations with a fixed \frank{}, and similarly,
    we will observe a lower \frank{} for low-rank approximations with
    a fixed accuracy.

Along the same line of reasoning, for a fixed kernel on a fixed domain,
when the point sets become denser, we should expect the
\frank{} to remain unchanged for a fixed precision. The result for
\frank{}s turns out to be in perfect agreement with the observations in
practical situations on \mrank{}s, assuming there are sufficiently many points to make the \mrank{} visible before reaching a given precision. 

We now turn to the case when $d$ is large. Because we have assumed $x_i$ and $y_i$ to be in $[0,D/\sqrt{d}]$, by concentration of measure, the values of $\|\vecx-\vecy\|_2^2$ will fall into a small-sized subinterval of $[0, D^2]$ with high probability. Therefore, we are interested in quantifying this probabilistic error bound.

   \begin{theorem}
        \label{thm:cheb1_prob}
        Suppose the \Cref{assumption:function_and_domain} and the \Cref{assumption:analytic} hold, and points $\vecx$ and $\vecy$ are sampled under the probability distribution involving $D$, $\sigma_d$ and $E_d$ in \Cref{assumption:data_distribution}.   
        We define function $\tilde{f}(x- E_d^2) = f(x)$. Then, $\tilde{f}$ is analytic in $[-E_d^2, D^2-E_d^2]$, with the parameter of its transformed Bernstein ellipse to be $\tilde{\rho}_D^2 > 1$, and $|\tilde{f}(x)| \le \tilde{C}_D$ inside the ellipse.  Defining the same error $
        \epsilon_n$ as in \Cref{thm:cheb1}
         \begin{equation}
            \epsilon_n(D, \delta) = \calK(\vecx,\vecy) - \sum_{i=1}^{R} g_i(\vecx)h_i(\vecy),\,\,\text{with } R = \binom{n+d+2}{d+2}
        \end{equation}
        we obtain that for $0 < \delta < D$, with probability at least 
         \begin{equation}
         \label{eq:probability}
         1- 2\exp\left(\frac{-\delta^4 \, d}{2\sigma_d^2 \, d + 8D^2 \delta^2 / 3}\right)
         \end{equation}
        the error can be bounded by
        $$
        |\epsilon_n(D, \delta)| \le \frac{2C_D\delta^2}{D^2(\tilde{\rho}_D^2-\tilde{\rho}_D^{-2}) - \delta^2}\left(\frac{D^2 (\tilde{\rho}_D^2 - \tilde{\rho}_D^{-2})}{\delta^2}\right)^{-n}
        $$
         And with the same probability, the distance of a sampled pair will fall into the following interval
        $$
        \|\vecx-\vecy\|_2^2 \in [E_d^2 - \delta^2, E_d^2 + \delta^2]
        $$
    \end{theorem}
    The proof of \Cref{thm:cheb1_prob} can be found in \SeCref{pf:theorem2_3}.

        In \Cref{thm:cheb1_prob}, as $d \to \infty$, $\delta$ needs to decrease with $d$ to maintain the same probability. If we choose $\delta = \left(\frac{C}{d}\right)^{1/4}$ with $C$ being a very large number, then the probability remains close to 1 because $\sigma_d^2 = \Theta(\frac{1}{d})$. Moreover, we can keep $\epsilon_n$ small while reducing $n$, because $\delta \to 0$. This means that for sufficiently large $d$ and for a given error, $n$ goes down as $d$ increases. Asymptotically, $n$ reaches 0, and the \frank{} reaches 1. On the other hand, for a fixed $n$, the error bound decreases when $d$ increases. 

Note that $2\delta^2$ is the size of the subinterval where the values of $\norm{\vecx-\vecy}_2^2$'s fall into with probability given by \cref{eq:probability} and, by concentration of measure, with the same probability, the interval size $2\delta^2$ shrinks with increasing $d$. This is consistent with what we have discussed that $\delta$ needs to decrease with $d$ to maintain the same probability. 

The analytic assumption in \Cref{thm:cheb1} and \Cref{thm:cheb1_prob} is very strong because many RBFs are not infinitely differentiable when the domain contains zero. However, most RBFs of practical interest are $q$-times differentiable. In the following theorem, we weaken the analytic assumption to a finite-smoothness assumption and compute the corresponding error bound.

\begin{theorem}
    \label{thm:cheb2}
    Suppose the \Cref{assumption:function_and_domain} and the \Cref{assumption:finite_smooth} hold.
    Then for $n > q$, the kernel $\calK$ can be approximated
    in the $L_\infty$ norm by a low-rank kernel $\widetilde{\calK}$
    of \frank{} at most $R(n, d) = \binom{n+d+2}{d+2}$
        \begin{equation*}
            \calK(\vecx,\vecy) = \sum_{i=1}^{R} g_i(\vecx)h_i(\vecy)
            + \epsilon_n = \widetilde{\calK}(\vecx,\vecy) + \epsilon_n
        \end{equation*}
     where $\{g_i\}_{i=1}^R$ and $\{h_i\}_{i=1}^R$ are two sequences
     of $d$-variable polynomials. And the error term $\epsilon_n =
     \epsilon_n(V_q, D, q)$ is bounded as
    \begin{equation}
    \label{eq:error_finite}
    |\epsilon_n(V_q, D, q)| \le  \frac{2V_qD^{2q}}{\pi q [2(n- q)]^q}
    \end{equation}
\end{theorem}

{\bf Remark.} We can weaken the assumption of $f^{(q)}$ having bounded total variation to $f^{(q-1)}$ being Lipschitz continuous, and this does not impose assumptions on $f^{(q)}$. With this weaker assumption, we obtain the same error rate $O(n^{-q})$; however, the trade off is the absent of explicit constants in the upper bound \cref{eq:error_finite}. 

The proof of \Cref{thm:cheb2} can be found in \SeCref{pf:theorem_2_4_2_5}.

Compared to \Cref{thm:cheb1}, the convergence rate slows down from a nice geometric convergence rate $O(\rho_D^{-2n})$ to an algebraic convergence rate $O(n^{-q})$. Each time the function becomes one derivative smoother ($q$ increased by 1), the convergence rate will also become one order faster. The domain diameter $D$ affects the error bound by $D^{2q}$, where $q$ represents the smoothness of the function. For a sub-domain with diameter $\widetilde{D}$, it is straightforward to obtain that the error is bounded by $\frac{2V_q\widetilde{D}^{2q}}{\pi q [2(n- q)]^q}$, and for a fixed $n$, a decrease in $\widetilde{D}$ will reduce the error. 

We also consider the phenomenon of concentration of measure and present the probabilistic result in the following theorem.
$x_i$ and $y_i$ are i.i.d.\ random variables, with $|x_i\sqrt{d}| < D$ and $|y_i\sqrt{d}| < D$, and have their second moments exist.
\begin{theorem}
    \label{thm:cheb2_prob}
    Suppose the \Cref{assumption:function_and_domain} and the \Cref{assumption:finite_smooth} hold.
    We further assume $\vecx$ and $\vecy$ are sampled under the probability distribution involving $D$, $E_d$, and $\sigma_d$ in \Cref{assumption:data_distribution}. Defining the same $\epsilon_n$ as in \Cref{thm:cheb2}
        \begin{equation*}
            \epsilon_n(V_q, \delta, q) = \calK(\vecx,\vecy) - \sum_{i=1}^{R} g_i(\vecx)h_i(\vecy), \quad
         \text{with } R = \binom{n+d+2}{d+2}
        \end{equation*}
    Then, for $0 < \delta < D$, we obtain the following bound
    $$
    |\epsilon_n(V_q, \delta, q)| \le  \frac{2V_q\delta^{2q}}{\pi q [2(n- q)]^q}
    $$
    with probability at least
    $$1- 2\exp\left(\frac{-\delta^4 \, d}{2\sigma_d^2 \, d + 8D^2 \delta^2 / 3}\right)$$
\end{theorem}
The proof of \Cref{thm:cheb2_prob} can be found in \SeCref{pf:theorem_2_4_2_5}.

Up to now, we have only considered a single parameter $D$ that characterizes the domain, to make the error bound more informative as in response to subtler changes of the domain, we also consider the diameters of the target domain $D_\vecx$ and of the source domain $D_\vecy$.
The following theorem nicely quantifies the influences of $D_\vecx$ and $D_\vecy$ on the error. Our result theoretically offers critical insights and motivations for many algorithms that take advantage of the low-rank property of sub-matrices, where these sub-matrices usually relate to data clusters of small diameters. 

\begin{theorem}
    \label{thm:fourier_taylor}
     Suppose the \Cref{assumption:function_and_domain} hold, and there are $D_\vecx < D$ and $D_\vecy < D$, such that $\|\vecx_i-\vecx_j\|_2\le D_\vecx$ and $\|\vecy_i-\vecy_j\|_2\le D_\vecy$.

    Let $f_p(x) = \sum_{n} \mathcal{T} \circ f(x+4nD^2)$ be a $4D^2$-periodic extension of $f(x)$, where $\mathcal{T}(\cdot)$\footnote{see details in \cite{Boyd2002}} is 1 on $[-D^2, D^2]$ and smoothly decays to 0 outside of the interval. We assume that $f_p$ and its derivatives through $f_p^{(q-1)}$ are continuous, and the $q$-th derivative is piecewise continuous with the total variation over one period bounded by $V_q$.

    Then, for $M_f, M_t > 0$ with $9M_f \le M_t$, the kernel $\calK$ can be approximated by a low-rank kernel $\widetilde{\calK}$ of rank at most $R(M_f, M_t, d) = 4M_f \binom{M_t+d}{d}$
        \begin{equation*}
            \calK(\vecx,\vecy) = \sum_{i=1}^{R} g_i(\vecx)h_i(\vecy)
            + \epsilon_{M_f, M_t} = \widetilde{\calK}(\vecx,\vecy) + \epsilon_{M_f, M_t}
        \end{equation*}
    And the error $\epsilon_{M_f, M_t} = \epsilon_{M_f, M_t}(D_\vecx, D_\vecy, q, \rho)$ is bounded by 
    $$
    |\epsilon_{M_f, M_t}(D_\vecx, D_\vecy, q, \rho)| \le \norm{f}_\infty\left(\frac{D_\vecx D_\vecy}{D^2}\right)^{M_t+1} +\frac{V_q}{\pi q} \left(\frac{2D^2}{\pi M_f}\right)^q
    $$
\end{theorem}
The proof of \Cref{thm:fourier_taylor} can be found in \SeCref{pf:theorem_2_6}.

In contrast to the previous theorems where the domain information only enters the error as $D$, in \Cref{thm:fourier_taylor}, the diameters of the source domain $D_\vecy$ and the target domain $D_\vecx$ also appear in error. The form $\left(\frac{D_\vecx D_\vecy}{D^2}\right)^{M_t+1}$ suggests that a decrease in $\frac{D_\vecx D_\vecy}{D^2}$ will reduce the error, which can be achieved when either the source or the target domain has a smaller diameter. 
This property has motivated people to approach matrix approximation problems by identifying low-rank blocks in a matrix, which is partially achieved by partitioning the data into clusters of small diameters.

The \frank{} still remains a polynomial growth and it grows as $R = O(d^{M_t})$, when $M_f$ and $M_t$ are fixed and
$d \rightarrow \infty$. $M_f$ represents the Fourier expansion order of
$f$, and each term in the expansion is further expanded into Taylor
terms up to order $M_t$. We assumed $M_t$ to be the same across all the
Fourier terms for simplicity. If we decrease the Taylor order $M_t$
with increasing Fourier order to preserve more information of low-order Fourier terms, then a lower error bound can
be attained for the same \frank{}.

{\bf Remark.}
We summarize the assumptions, error bounds and \frank{}s of the 
theorems in \Cref{tab:methods_comparison}, and discuss the similarities and
differences in the \frank{} and the error bound. We
refer to \Cref{thm:cheb1} and \Cref{thm:cheb2} as the Chebyshev approach and 
\Cref{thm:fourier_taylor} as the Fourier-Taylor approach based on their
proof techniques. The \frank{} is determined by the data dimension and the expansion order, and it is a power of the dimension, where the power is the expansion order and is different in the Chebyshev approach and the Fourier-Taylor approach. The error bounds quantify the influences from the expansion order and the domain diameter: a higher expansion order reduces the error bound, so does a smaller domain diameter. The domain diameter occurs as a single parameter $D$ in the Chebyshev approach but as $D_\vecx$, $D_\vecy$ and $D$ in the Fourier-Taylor approach. 

From the practical viewpoint, the absence of exponential growth for the \frank{} agrees with the practical situation where people observe lower \mrank{}s for high dimension data. And, the fact that decreasing $D_\vecx$ or $D_\vecy$ reduces the error is also in agreement with practice and moreover, it provides an insight of why point clusterings followed by local interpolations often leads to a more memory efficient approximation.

\begin{table}[htbp]
    \scriptsize
    \caption{Theorem Summary}
    \centering
    \begin{adjustbox}{angle=90}

    \begin{tabular}{ p{14mm}| p{30mm}|p{30mm}|p{30mm}|p{30mm}|   p{42mm} }
        \toprule
        \multirow{ 2}{*}{{\bf Approach}}  & \multicolumn{4}{|c|}{\begin{tabular}{@{}c@{}}Chebyshev expansion \\ + Exact expansion of $\norm{\vecx - \vecy}^{2l}$\end{tabular}}  & \begin{tabular}{@{}c@{}} Fourier expansion \\ + Taylor expansion of $\exp{(\imath\vecx^T\vecy)}$ \end{tabular}\\
        \cmidrule(){2-6}
        &Deterministic (analytic) & Probabilistic (analytic)& Deterministic (finite smoothness)& Probabilistic (finite smoothness) & Deterministic (finite smoothness)\\
            \midrule
            {\bf Condition}
            &
            \begin{tabitemize}
                \item \Cref{assumption:function_and_domain}
                \item \Cref{assumption:analytic}
            \end{tabitemize}
            &
            \begin{tabitemize}
             \item \Cref{assumption:function_and_domain}
            \item \Cref{assumption:analytic}
            \item $\tilde{f}(x- E_d^2) = f(x)$ with parameter of its transformed Bernstein ellipse to be $\tilde{\rho}_D^2 > 1$
             \item \Cref{assumption:data_distribution}
            \end{tabitemize}
            &
            \begin{tabitemize}
                \item \Cref{assumption:function_and_domain}
                \item \Cref{assumption:finite_smooth}
            \end{tabitemize}
            &
            \begin{tabitemize}
                \item \Cref{assumption:function_and_domain}
                \item \Cref{assumption:finite_smooth}
                \item \Cref{assumption:data_distribution}
            \end{tabitemize}
            &
            \begin{tabitemize}
            	\item \Cref{assumption:function_and_domain}
                \item $f_p(x)$ is the $4D^2$-periodic extension of $f$ and $\|f(x) - f_p(x)\|_{\infty, [-D^2, D^2]} = 0$
                \item The first $q-1$ derivatives of $f_p$ are continuous, and the $q$-th derivative on $[-D^2,D^2]$ is piece-wise continuous with bounded total variation $V_q$
                \item $9M_f \le M_t$
            \end{tabitemize}
        \\
        \midrule
        {\bf Error} & {$\dfrac{2C_D\rho_D^{-2n}}{\rho_D^2- 1} $}  & $\dfrac{2C_D\delta^2}{D^2\Big(\tilde{\rho}_D^2-\tilde{\rho}_D^{-2}\Big) - \delta^2} \times$ \newline $\left(\dfrac{D^2 \Big(\tilde{\rho}_D^2 - \tilde{\rho}_D^{-2}\Big)}{\delta^2}\right)^{-n}$ with probability at least  $Pr(\delta, D, \sigma_d, d)$ for $0 < \delta < D$
		&
            {$\dfrac{2V_qD^{2q}}{\pi q [2(n- q)]^q} $} & $\dfrac{2V_q\delta^{2q}}{\pi q [2(n- q)]^q}$ with probability at least $Pr(\delta, D, \sigma_d, d)$ for $0 < \delta < D$
        & $    \norm{f}_\infty\left(\dfrac{D_\vecx D_\vecy}{D^2}\right)^{M_t+1} +\dfrac{V_q}{\pi q} \left(\dfrac{2D^2}{\pi M_f}\right)^q $  \\
        \midrule
        {\bf Rank} & \multicolumn{4}{|c|}{$ \binom{n+d+2}{d+2}$}&
        $4M_f \binom{M_t+d}{d} $\\
        \midrule
        {\bf Notation}
            &
            \multicolumn{4}{|p{200pt}|}{
           $n$: Chebyshev expansion order \newline
           $D$: $\norm{\vecx-\vecy}_2 < D$ \newline
           $E_d^2 = \sum_{i=1}^d {\bf E}[(x_i-y_i)^2]$ \newline
           $\sigma_d^2 = \sum_{i = 1}^d {\bf Var}[(x_i - y_i)^2]$ \newline
           $Pr(\delta, D, \sigma_d, d) = 1- 2\exp\left(\dfrac{-\delta^4 \, d}{2\sigma_d^2 \, d + 8D^2 \delta^2 / 3}\right)$}
           &
           $D_\vecx$: $\norm{\vecx_i-\vecx_j}_2 < D_\vecx$ \newline
           $D_\vecy$: $\norm{\vecy_i-\vecy_j}_2 < D_\vecy$ \newline
           $M_f$: Fourier expansion order\newline
           $M_t$: Taylor expansion order\\
        \bottomrule
    \end{tabular}
    \end{adjustbox}
    \label{tab:methods_comparison}
\end{table}

   \section{Theorem Proofs}
   \label{sec:thm_proofs}
    In this section, we prove the theorems in \Cref{sec:theorem}. All the proofs consist of three components: separating $\calK(\vecx, \vecy)$ into a finite sum of products of real-valued functions $h_i(\vecx)g_i(\vecy)$, counting the terms to obtain an upper bound for the \frank{}, and calculating the error bound. Similar techniques can be found in \cite{micchelli1984interpolation, schaback2008limit, Wathen2015, zwicknagl2009power}. We describe the high-level procedure of the separation step; the rest steps should be straightforward. 

    In the proofs of \Cref{thm:cheb1} and \Cref{thm:cheb2}, the separable form was obtained by first expanding the kernel into polynomials of $z = \norm{\vecx-\vecy}^2$ of a certain order to settle the error bound, and then expanding the terms $\norm{\vecx-\vecy}^{2l}$. The key advantage of this approach has been discussed at the beginning of \Cref{sec:theorem}. We seek approximation theorems in 1D that provide optimal convergence rate and explicit error bounds. Chebyshev theorems (Theorem 8.2 and Theorem 7.2 in \cite{Trefethen2012}) are ideal choices. Analogous results also exist, \emph{e.g.}, the classic Bernstein and Jackson's approximation theorems (page 257 in \cite{bernstein1912recherches}), but the downside is that they only provide an error rate rather than an explicit formula, and moreover, they will not improve our results or simplify the proofs.

    In the proof of \Cref{thm:fourier_taylor}, the separable form was obtained by first applying a Fourier expansion on $\calK$ to separate the cross term $\exp(\vecx^T\vecy)$, then applying a Taylor expansion on the cross term.
    
	Before stating the detailed proofs, we introduce some notations that will be used. \\
    \textbf{Notations.}
    For multi-index $\vecalpha = [\alpha_1, \cdots, \alpha_d] \in \mathbb{N}^d$ and vector $\vecx = [x_1, \cdots, x_d] \in \mathbb{R}^d$, we define
    $
    |{\vecalpha}| = \alpha_1 + \alpha_2 + \cdots + \alpha_d
    $,
    $
    \vecx^{\vecalpha} = x_1^{\alpha_1}x_2^{\alpha_2}\cdots x_d^{\alpha_d}
    $
    and the multinomial coefficient with $|\vecalpha| = m$ to be
    $
    \binom{m}{\vecalpha} = \frac{m!}{\alpha_1!\alpha_2!\cdots\alpha_d!}.
    $

\subsection{Proof of Theorem 2.1}
    \label{pf:theorem2_1}
	We first introduce a lemma on the identity of binomial coefficients.
    \begin{lemma}
    \label{lem:binomial_identity}
    For $d \in \mathbb{Z}^{+}$ and $m \in \mathbb{Z}$, the following identity holds:
    \begin{equation*}
    	\sum_{k=0}^m \binom{k+d}{d} = \binom{m+1+d}{d+1}
    \end{equation*}
    \end{lemma}
    \begin{proof}
    The proof can be done by induction and follows that from Lemma 2.4 in \cite{wendland2004scattered}.
    \end{proof}

\begin{proof}
    The proof consists of two components. First, we map the domain of $f$ to $[0, 2]$ (for the convenience of the proof) and approximate $f$ with a Chebyshev polynomial, and this settles the error. Second, we further separate terms $\norm{\vecx-\vecy}^2$ in the polynomial and count the number of distinct terms to be an upper bound of the \frank{}.

    \begin{itemize}[leftmargin=*]

        \item[] \textit{Approximation
by Chebyshev polynomials.}
        We first linearly map the domain of $f$ to $[0,2]$ and denote the new function as $\tilde{f}$:
            \begin{equation}
                \label{eq:linear_transform}
                \calK(\vecx, \vecy) =  f(\norm{\vecx-\vecy}_2^2) = \tilde{f}\left(\frac{2}{D^2} \norm{\vecx-\vecy}_2^2\right) = \tilde{f}(z)
            \end{equation}
            Because $\norm{\vecx - \vecy}^2 \in [0, D^2]$, it follows that $z \in [0, 2]$. From our assumptions, $\tilde{f}$ is analytic in $[0, 2]$ and is analytically continuable to the open Bernstein ellipse  with parameter ${\rho_D^2}$ (consider a shifted ellipse).

            According to Theorem 8.2 in \cite{Trefethen2012} that follows from \cite{little1984eigenvalues}, for $n \ge 0$, we can approximate $\tilde f$ by its Chebyshev truncations $\tilde{f}_n$ in the $L_\infty$ norm with error
            \begin{align}
                |\epsilon_n| \le \frac{2C_D \rho_D^{-2n}}{ \rho_D^2 - 1} 
            \end{align}
            and
            \begin{equation}
                \label{eq:chebyshev_polynomial}
                \begin{split}
                    {\tilde f}_n(z) =  \sum_{k=0}^{n} c_k T_k(z) + \epsilon_n
                \end{split}
            \end{equation}
            where $c_k = \frac{2}{\pi} \bigintsss_{-1}^1 \frac{{\tilde f}(z)T_k(z)}{\sqrt{1-z^2}}dz$, and
            $T_k(z)$ is the Chebyshev polynomial of the first kind of degree $k$ defined by the relation:
            \begin{align}
                T_k(x) = \cos(k\theta), ~\text{with}~x = \cos(\theta).
            \end{align}

            Rearranging the terms in \cref{eq:chebyshev_polynomial} we obtain a polynomial of $z = \norm{\vecx-\vecy}^2$:
            \begin{align}
                \label{eq:polynomial_k}
                \calK(\vecx, \vecy) = {\tilde f}\left(\frac{2}{D^2} \norm{\vecx-\vecy}_2^2\right) = \sum_{k=0}^n\frac{a_k}{D^{2k}}\norm{\vecx-\vecy}^{2k}+ \epsilon_n
            \end{align}
            where $a_k$ depends on $c_k$ but is independent of $\vecx$ and $\vecy$.

        \item[] \textit{Separable form.}
            We separate each term $\norm{\vecx - \vecy}^{2l}$ in \cref{eq:polynomial_k} into a finite sum of separate products:
            \begin{equation}
                \label{eq:distance_expansion}
                \begin{split}
                    \|\vecx - \vecy\|^{2l} &= \sum_{k=0}^{l} {l \choose k}(-2)^{l-k}   \sum_{j=0}^k {k \choose j} \norm{\vecx}^{2j} \norm{\vecy}^{2(k-j)} \left(\sum_{i=1}^d x_i y_i\right)^{l-k} \\
                    &= \sum_{k=0}^{l} \sum_{j=0}^k\sum_{|\vecalpha| = l-k} C_{l, k, \vecalpha} \left( \norm{\vecx}^{2j}  \vecx^{\vecalpha} \right) \left(\norm{\vecy}^{2(k-j)}\vecy^{\vecalpha} \right)
                \end{split}
            \end{equation}
            where $ C_{l, k, \vecalpha} = (-2)^{l-k}  \binom{l}{k} \binom{k}{j} \binom{l-k}{\vecalpha}$.
            Substituting \cref{eq:distance_expansion} into \cref{eq:polynomial_k}, we obtain a separable form of $\calK$:
            \begin{equation}
                \label{eq:kernel_seperate_form}
                \calK(\vecx, \vecy) =  \sum_{l=0}^n\sum_{k=0}^{l}\sum_{j=0}^k \sum_{|\vecalpha| = l-k} D_{l, k, \vecalpha} \left( \norm{\vecx}^{2j}  \vecx^{\vecalpha} \right) \left(\norm{\vecy}^{2(k-j)}\vecy^{\vecalpha} \right)
                + \epsilon_n 
            \end{equation}
            where $D_{l, k, \vecalpha} = \frac{a_l}{D^{2l}}(-2)^{l-k}  \binom{l}{k} \binom{k}{j} \binom{l-k}{\vecalpha}$ is a constant independent of $\vecx$ and $\vecy$.
            Therefore, the \frank{} of $\calK$ can be upper bounded by the total number of separate terms:
            \begin{align*}
                \sum_{l=0}^n \sum_{k=0}^{l} (k+1)\binom{l-k + d-1}{d - 1}
                                = \binom{n+d+2}{d+2}
            \end{align*}
where the equality follows from the result in \Cref{lem:binomial_identity}.
            To summarize, we have proved that $\calK(\vecx, \vecy)$ can be approximated by the separable form in \cref{eq:kernel_seperate_form} in the $L_\infty$ norm with rank at most
            \begin{equation}
                \label{eq:rank_cheb}
                R(n, d) = \binom{n+d+2}{d+2}
            \end{equation}
        and approximation error
        \begin{equation}
            |\epsilon_n(D)| \le  \frac{2C_D \rho_D^{-2n}}{\rho_D^2- 1}
        \end{equation}
    \end{itemize}
\end{proof}

\subsection{Proof of Corollary 2.2}
\label{pf:Corollary2_2}
\begin{proof}
For a fixed kernel function and fixed $n$, we define two constants $c_1 = \frac{\rho_D^2- 1}{2C_D}$ and $c_2 = \log \rho_D^2$.
Then, the truncation error $\epsilon$ can be rewritten as
\[
\epsilon = \frac{2C_D \rho_D^{-2n}}{\rho_D^2- 1}  = \frac{e^{- n c_2}}{c_1}\] 
and equivalently,
\begin{equation}
    n = \frac{-\log c_1\epsilon}{c_2}
\end{equation}
We relate \frank{} $R$ to error $\epsilon$ and dimension $d$. When $d \geq n+2$, we obtain,
\begin{equation}
    \begin{split}
        \label{eq:rank_d_error}
        R &=  \binom{n+d+2}{d+2}
        \le \frac{2^n d^n}{n!}
        = c_n  d^\frac{-\log c_1 \epsilon}{c_2} \\
    \end{split}
\end{equation}
where $c_n = \frac{2^n}{n!}$ is a constant for a fixed $n$. Therefore, an $\epsilon$ error is achievable with the function rank $R$ proportional to $d^{\frac{-\log c_1\epsilon}{c_2}}$.
\end{proof}

\subsection{Proof of Theorem 2.3}
\label{pf:theorem2_3}
\begin{proof}
	We consider the concentration of measure phenomenon and apply concentration inequalities to obtain a probabilistic error bound. The proof mostly follows the proof of Theorem 2.1, and we will focus on computing the error bound for a smaller domain. 
    
 	To simplify the proof, we consider a function $\tilde{f}$ that is shifted by $E_d^2$ such that 
    \begin{equation*}
    	f(\norm{\vecx-\vecy}_2^2) = {\tilde f}(\norm{\vecx-\vecy}_2^2 - E_d^2)
    \end{equation*}
    and we will see later that this shift ensures the inputs for $\tilde{f}$ to fall into an interval that centers around 0 with some probability. 
    $\tilde{f}$ inherits the analyticity of $f$; therefore, it is analytic on $[-E_d^2, D^2 - E_d^2]$, and can be analytically extended to a transformed Bernstein Ellipse with parameter $\tilde{\rho}_D^2$. 
    
    Let us denote $z_i = (x_i-y_i)^2 - {\bf E}[(x_i-y_i)^2]$ and we will shortly apply concentration inequality to $z_i$. With the assumptions that $x_i\sqrt{d}$ and $y_i\sqrt{d}$ are i.i.d.\ random variables where $|x_i\sqrt{d}| < D$ and $|y_i\sqrt{d}| < D$, it follows that the $z_i$ are statistically independent with mean zero and are bounded by $\frac{4D^2}{d}$. By applying Bernstein's inequality \cite{bernsteinprobability} on the sum of the $z_i$, we conclude that for $\delta \ge 0$,
    \begin{equation}
    \label{eq:concentration_prob}
        P\left(| \norm{\vecx - \vecy}_2^2  -E_d^2| \le \delta^2\right)   \ge  1- 2\exp\left(\frac{-\delta^4 \, d}{2\sigma_d^2 \, d + 8D^2 \delta^2 / 3}\right)
    \end{equation}
        where $E_d^2 = \sum_{i=1}^d {\bf E}[(x_i-y_i)^2]$ is a constant. In other words, $\norm{\vecx - \vecy}_2^2 \in [E_d^2 - \delta^2, E_d^2 + \delta^2]$ with probability at least
        $$1- 2\exp\left(\frac{-\delta^4 \, d}{2\sigma_d^2 \, d + 8D^2 \delta^2 / 3}\right)$$
        This also means that with the same probability in \cref{eq:concentration_prob}, the inputs for $\tilde{f}$ will fall into the interval $[-\delta^2, \delta^2]$. 
        
        Therefore, for a probability associated with $\delta$, we can turn to considering $\tilde{f}$ on the domain $[-\delta^2, \delta^2]$. We assume that $\tilde{f}$ is analytically extended to a transformed Bernstein ellipse with parameter $\rho_\delta^2$, with the value of $\tilde{f}(z)$ inside the ellipse bounded by $C_\delta$. Following the same argument in the proof for \Cref{thm:cheb1}, we obtain that for $\delta > 0$ and with probability in \cref{eq:concentration_prob}, the approximation error for $\vecx$ and $\vecy$ sampled from the above distribution is bounded by
    \begin{equation}
    \label{eq:error_prob}
        |\epsilon_n| \le \frac{2C_\delta}{\rho_\delta^2- 1}\rho_\delta^{-2n}
    \end{equation}
    This sharper bound can be achieved with the same function rank as in \cref{eq:rank_cheb} and with the same low-rank representation as in \cref{eq:kernel_seperate_form} except for coefficients.  
    
    Next, we rewrite the upper bound in \cref{eq:error_prob} with the parameters $\tilde{\rho}_D$, $\tilde{C}_D$, and $\delta$. If we linearly map the domain of $\tilde{f}$ from $[-\delta^2, \delta^2]$ to [-1,1], then the Bernstein ellipse with parameter $\tilde{\rho}_D^2$ will be scaled by $\frac{1}{\delta^2}$. We seek the largest $\rho_\delta^2$ such that the Bernstein ellipse with parameter $\rho_\delta^2$ will be contained in the transformed Bernstein ellipse with parameter $\tilde{\rho}_D^2$. In that case, the lengths of their semi-minor axises match and the largest $\rho_\delta^2$ satisfies 
    \begin{equation}
    \rho_\delta^2 - \rho_\delta^{-2} = \frac{D^2}{\delta^2} \left(\tilde{\rho}_D^2 - \tilde{\rho}_D^{-2}\right)
    \end{equation}
    and we obtain $\rho_\delta^2 = \frac{D^2}{\delta^2} (\tilde{\rho}_D^2 - \tilde{\rho}_D^{-2}) + \left(\frac{D^4}{4\delta^4}(\tilde{\rho}_D^2 - \tilde{\rho}_D^{-2})^2+1\right)^{\frac{1}{2}}$. In the special case where $\delta^2 = D^2$, $\rho_\delta^2 = \tilde{\rho}_D^2$, we recover the error bound $\frac{2C_D}{\tilde{\rho}_D^2- 1}\tilde{\rho}_D^{-2n}$. To simplify the bound, we use the relation that $\rho_\delta^2 > \frac{D^2}{\delta^2}(\tilde{\rho}_D^2 - \tilde{\rho}_D^{-2})$. Substituting this into \cref{eq:error_prob}, along with the fact that $C_\delta \le C_D$, we obtain
    \begin{equation}
    \label{eq:error_prob_thm2_3}
        |\epsilon_n| \le \frac{2C_D\delta^2}{D^2(\tilde{\rho}_D^2-\tilde{\rho}_D^{-2}) - \delta^2}\left(\frac{D^2 (\tilde{\rho}_D^2 - \tilde{\rho}_D^{-2})}{\delta^2}\right)^{-n}
    \end{equation}
    
  Therefore, the function rank related to error $\epsilon_n$ remains $\binom{n+d+2}{d+2}$, and we have proved our result.
\end{proof}

\subsection{Proof of Theorem 2.4 and Theorem 2.5}
\label{pf:theorem_2_4_2_5}
\begin{proof}
    The proof follows the same steps as that in \Cref{thm:cheb1} and \Cref{thm:cheb1_prob};
    we only need to establish that the error term in the Chebyshev expansion is bounded by $\frac{2D^{2q}V_q}{\pi q [(n- q)]^q} $.
    Consider \cref{eq:linear_transform}.
    Because $f^{(q)}$ is piecewise continuous with its total variation on
    $[0, D^2]$ bounded by $V_q$, it follows that ${\tilde f}^{(q)}$ in \cref{eq:linear_transform} is piecewise continuous on $[0,2]$, with its total variation on $[0,2]$ bounded as follows
    \[
        V\left(\frac{\dif^q {\widetilde f}}{\dif x^q}\right) = V\left(\frac{D^{2q}}{2^q}\frac{\dif^q f}{\dif^q x^q}\right) = \frac{D^{2q}}{2^q}  V\left(\frac{\dif^q f}{\dif^q x^q}\right)  \le \frac{D^{2q}}{2^q} V_q
    \]
    Therefore, by Theorem 7.2 in \cite{Trefethen2012}, for $n > q$, the order-$n$ Chebyshev expansion ${\tilde f}_n$ approximates $\tilde f$ in the $L_\infty$ norm with error bounded by 
    \[
        |\epsilon_n| \le \frac{2V_q({\tilde f})}{\pi q (n- q)^q} \le \frac{2D^{2q}V_q}{\pi q (2(n- q))^q}
    \]
    The rest of the proof is identical to that of \Cref{thm:cheb1} for the deterministic result, and identical to that of \Cref{thm:cheb1_prob} for the probabilistic result.
\end{proof}

\subsection{Proof of Theorem 2.6}
\label{pf:theorem_2_6}
    We first introduce a lemma concerning the \frank{} of complex functions.
    \begin{lemma}
        \label{lem:complex-to-real-app}
        If a real-valued function $\calK$ can be approximated by
        two sequences of complex-valued functions,
        \textit{i.e.},
        $$
        \left|
        \calK(\vecx,\vecy) - \sum_{i=1}^{R_c}
        \Psi_i(\vecx)\Phi_i(\vecy)
        \right| \leq \epsilon,
        \quad \vecx\in\Omega_\vecx,\vecy\in\Omega_\vecy
        $$
        where $\{\Psi_i(\vecx)\}_{i=1}^{R_c}$
        and $\{\Phi_i(\vecy)\}_{i=1}^{R_c}$
        are complex-valued functions,
        then there exist two sequences of real-valued functions,
        $\{g_i(\vecx)\}_{i=1}^R$
        and $\{h_i(\vecy)\}_{i=1}^R$,
        such that for $R=2R_c$,
        $$
        \left|
        \calK(\vecx,\vecy) - \sum_{i=1}^R g_i(\vecx)h_i(\vecy)
        \right| \leq \epsilon,
        \quad \vecx\in\Omega_X,\vecy\in\Omega_Y
        $$
    \end{lemma}
    \begin{proof}
        Let $\real{\cdot}$ and $\imag{\cdot}$ denote the real and imaginary part
        of a complex value, respectively.
        For each term, $\Psi_i(\vecx)\Phi_i(\vecy)$,
        we rewrite it as
        \begin{equation}
            \begin{split}
                \Psi_i(\vecx)\Phi_i(\vecy)
                = & \left(
                \real{\Psi_i(\vecx)}\real{\Phi_i(\vecy)}
                - \imag{\Psi_i(\vecx)}\imag{\Phi_i(\vecy)}
                \right)\\
                &+ \imath
                \left(
                \real{\Psi_i(\vecx)}\imag{\Phi_i(\vecy)}
                + \imag{\Psi_i(\vecx)}\real{\Phi_i(\vecy)}
                \right)
            \end{split}
        \end{equation}

        We can then construct the sequences of real-valued functions as follows
        \begin{equation}\label{eq:lem-consturct}
            \left\{
                \begin{aligned}
                    &g_{2i-1}(\vecx) = \real{\Psi_i(\vecx)}, g_{2i}(\vecx) = -\imag{\Psi_i(\vecx)}\\
                    &h_{2i-1}(\vecy) = \real{\Phi_i(\vecy)}, h_{2i}(\vecy) = \imag{\Phi_i(\vecy)}
                \end{aligned}
            \right., \quad i=1,2,\ldots,R_c
        \end{equation}
        The approximation error holds for the real-valued approximation:
        \begin{equation}
            \begin{split}
                \left|
                \calK(\vecx,\vecy) - \sum_{i=1}^R g_i(\vecx)h_i(\vecy)
                \right|
                \leq
                \left|
                \calK(\vecx,\vecy) - \sum_{i=1}^{R_c}
                \Psi_i(\vecx)\Phi_i(\vecy)
                \right|
                \leq \epsilon
            \end{split}
        \end{equation}
    \end{proof}

    We now start the proof for \Cref{thm:fourier_taylor}.

    \begin{proof}
        The proof consists of three major parts: derivation of a separable form for $\calK(\vecx,\vecy)$, analysis on
        the truncation error, and estimation of the number of separable
        terms. The first part is proceeded in three steps: Fourier expansion of the
        periodic input function, Taylor expansion of each Fourier
        component, and finalization on the overall separable form.

        We denote by $\Omega_\vecx$ the domain of $\vecx$, and
        $\Omega_\vecy$ the domain of $\vecy$, with their centers to be
        $\vecx_c$ and $\vecy_c$, respectively. To simplify the notations, we use $f(\cdot)$ to represent the periodic function $f_p(\cdot)$.

        \begin{itemize}[leftmargin=*]

            \item[] \textit{Fourier expansion.}
                Let the Fourier expansion of $f$ with error term
                $\epsilon_F$ be
                \begin{equation} \label{eq:thm-fourier-expan1}
                    f(z) = \sum_{j=-M_f}^{M_f} a_j \exp(\imath \omega j
                    z) + \epsilon_F
                \end{equation}
                where $a_j = \frac{1}{4D^2}
                \int_{-2D^2}^{2D^2}f(z)\exp(-\imath \omega j
                z)\,\mathrm{d}z$ is the Fourier coefficient and $\omega
                = \frac{2\pi}{4D^2}$ is a constant. Each
                Fourier coefficient can be bounded by the infinity
                norm of function $f(z)$, \emph{i.e.}, $|a_j|
                \le \norm{f}_\infty$. A detailed analysis of the error
                $\epsilon_F$ will be discussed in the second major
                part of the proof.  The fact that $\calK(\vecx,\vecy) =
                f(\norm{\vecx-\vecy}_2^2)$ is a function of $z=\norm{\vecx-\vecy}_2^2$ naturally requires a separation of $z$
                in order to proceed with the separation of
                $\calK(\vecx,\vecy)$. Adopting notations, $\vecrho_\vecx
                = \vecx - \vecx_c$,  $\vecrho_\vecy = \vecy - \vecy_c$
                and $\vecrho_c = \vecx_c - \vecy_c$, we rewrite $z =
                \norm{\vecx-\vecy}_2^2 = \norm{\vecrho_\vecx+\vecrho_c}^2
                + \norm{\vecrho_\vecy}^2 - 2\vecrho_\vecy^T\vecrho_c -
                2 \vecrho_\vecx^T\vecrho_\vecy$ and, therefore,
                \begin{equation} \label{eq:thm-fourier-expan2}
                    \exp(\imath \omega j z) =
                    \underbrace{\exp({\imath \omega
                    j \norm{\vecrho_\vecx + \vecrho_c}^2})}_\text{function
                    of $\vecx$ only}
                    \underbrace{\exp({\imath \omega
                    j (\norm{\vecrho_\vecy}^2-2 \vecrho_\vecy^T
                    \vecrho_c)})}_\text{function of $\vecy$ only}
                    \underbrace{\exp({-\imath \omega
                    j 2\vecrho_\vecx^T\vecrho_\vecy})}_\text{function of
                    $\vecx$ and $\vecy$}
                \end{equation}

            \item[] \textit{Taylor expansion.}
                The last term in \cref{eq:thm-fourier-expan2} still
                involves both $\vecx$ and $\vecy$ and needs to be further
                separated. We apply a Taylor expansion to this term,
                \begin{equation} \label{eq:taylor}
                    \begin{split}
                        \exp\left({-\imath \omega j
                        2\vecrho_\vecx^T\vecrho_\vecy}\right) &= \sum_{k
                        = 0}^{M_t} \frac{(-\imath \omega j 2
                        \vecrho_\vecx^T\vecrho_\vecy)^{k}}{k!} +
                        \epsilon_T(j) \\ & =   \sum_{k = 0}^{M_t}
                        \frac{(-\imath 2j \omega)^{k}}{k!}
                        \sum_{|\vecalpha|=k} \binom{k}{\vecalpha}
                        \vecrho_\vecx^{\vecalpha}
                        \vecrho_\vecy^{\vecalpha} + \epsilon_T(j)
                    \end{split}
                \end{equation}
                where $M_t$ is the order of the Taylor expansion,
                $\epsilon_T(j)$ is the corresponding truncation error,
                and the last equality adopts the multi-index notation
                introduced earlier.

        \item[] \textit{Separable form.}
        Combining \cref{eq:taylor}, \cref{eq:thm-fourier-expan2}
        and \cref{eq:thm-fourier-expan1}, we obtain
        \begin{equation} \label{eq:Fourier_Taylor_separate}
            \begin{split}
                f(z) = \sum_{j=-M_f}^{M_f} \sum_{k=0}^{M_t}
                \sum_{|\vecalpha| = k} h_{j, \vecalpha}(\vecx)
                g_{j, \vecalpha}(\vecy) + \epsilon
            \end{split}
        \end{equation}
        where
        \begin{equation}
            \begin{split}
                h_{j, \vecalpha}(\vecx) & = a_j \frac{(-\imath 2 j
                \omega)^k}{k!} \binom{k}{\vecalpha}
                \exp\left({\imath \omega j \norm{\vecrho_\vecx
                +\vecrho_c}^2}\right) \vecrho_\vecx^{\vecalpha} \text{
                    and}\\
                g_{j, \vecalpha}(\vecy) &= \exp\left({\imath
                \omega j (\norm{\vecrho_\vecy}^2-2
                \vecrho_\vecy^T \vecrho_c)}\right)
                \vecrho_\vecy^{\vecalpha}
            \end{split}
        \end{equation}
        are functions of $\vecx$ only and $\vecy$ only, respectively,
        and $\epsilon$ is the overall error
        \begin{equation}
            \epsilon = \sum_{j=-M_f}^{M_f} a_{j}
                  \exp \left({\imath \omega j
                  \norm{\vecrho_\vecx + \vecrho_c}^2}\right)
                  \exp\left({\imath \omega j
                  ( \norm{\vecrho_\vecy}^2 -
                  2 \vecrho_\vecy^T \vecrho_c)} \right)
                  \epsilon_T(j) + \epsilon_F
        \end{equation}
        A na\"ive bound on $\epsilon$ is given as,
        \begin{equation} \label{eq:FT-err}
            |\epsilon| \le \sum_{j=-M_f}^{M_f} |a_{j}| |\epsilon_T(j)|
            + |\epsilon_F| \le 2M_f \norm{f}_\infty 
            \max_j \left|\epsilon_T(j)\right| + |\epsilon_F|
        \end{equation}
        where the first inequality used the fact that the absolute
        values of both exponential terms are one.

        \item[] \textit{Error analysis.}
        According to \cref{eq:FT-err}, the total error consists of
        two parts, the truncation errors from the Taylor expansion and
        from the Fourier expansion. We consider first the Taylor expansion errors. Applying the Lagrange remainder form, we bound the Taylor part of the total error as
        \begin{equation}
            \begin{split}
                2M_f\norm{f}_\infty \max_j \abs{\epsilon_T(j)} & =
                2M_f\norm{f}_\infty \max_j \abs{ \frac{\left(-\imath
                \omega j 2 \vecrho_\vecx^T\vecrho_\vecy\right)^{
                    (M_{t}+1) }} { (M_{t}+1)!} } \\
                & \leq 2 M_f \norm{f}_\infty \abs{\frac{ \left( 2 \omega
                M_f \vecrho_\vecx^T \vecrho_\vecy \right)^{M_t+1}}{
                (M_t+1)! }} \\
                & \leq \frac{2 (e M_f)^{M_t+2}}{e^2 (M_t+1)^{M_t+1}}
                \left(\frac{ D_\vecx D_\vecy }{ D^2 }\right)^{ M_t+1 }\\
                & \leq \left(\frac{D_\vecx D_\vecy}{D^2
                }\right)^{M_t+1}
            \end{split}
        \end{equation}
	    where the second inequality adopts the inequality
	    $e(\frac{n}{e})^n \le n!$ with $e$ being the Euler's
	    constant, and the third inequality can be verified with
	    our assumption $9 M_f \leq M_t$.

	    We then consider the Fourier expansion errors. According to Theorem 2 in \cite{Giardina1972}, the truncation
        error of the Fourier expansion, $\epsilon_F$ can be
        bounded as follows
        \begin{equation}
            |\epsilon_F| \leq \frac{V_q}{\pi q(\omega M_f)^q} =
            \frac{V_q}{\pi q} \left(\frac{2D^2}{\pi M_f}\right)^q
        \end{equation}
        where $V_q$ is the total variation of the $q$-th derivative of
        $f(z)$ over one period.

        Therefore, the total error $\epsilon$ in
        \cref{eq:Fourier_Taylor_separate} can be bounded as
        \begin{equation}
            |\epsilon| \le \norm{f}_\infty \left( \frac{ D_\vecx D_\vecy
            }{ D^2 } \right)^{M_t+1} + \frac{V_q}{\pi q} \left( \frac{
                2D^2 }{ \pi M_f }\right)^q
        \end{equation}

        \item[] \textit{Rank computation.}
        Equation~\cref{eq:Fourier_Taylor_separate} is a separable form
        of $\calK(\vecx,\vecy)$ in its complex form with rank at most
        \begin{equation} \label{eq:rank}
            R_c = 2M_f \sum_{\ell=0}^{M_t} \binom{\ell+d-1}{d-1}
            = 2M_f \binom{M_t+d}{d}
        \end{equation}
         where the equality comes from \Cref{lem:binomial_identity}.
        By \Cref{lem:complex-to-real-app}, the kernel function can be
        approximated by two sequences of real-valued functions
        $\{g_i\}_{i=1}^R$ and $\{h_i\}_{i=1}^R$ with rank at most
        \begin{equation}
            R(M_f, M_t, d) = 2 R_c \le 4M_f \binom{M_t+d}{d}
        \end{equation}
        Note, when $M_f$ and $M_t$ are fixed and $d \rightarrow \infty$,
        the rank grows as $O(d^{M_t})$.
        \end{itemize}
    ~
    \end{proof}


\section{Optimality of the polynomial growth of the \frank{}}
\label{sec:optimal}
\Cref{thm:cheb1_corollary} shows that asymptotically, for a given error
$\epsilon$ and dimension $d$, the \frank{} needed for a low-rank representation to approximate an analytic function with error $\epsilon$ is proportional to $d^{\frac{-\log c_1\epsilon}{c_2}}$. We will show that up to some constant, this asymptotic rank has achieved the lower bound on the minimal number of interpolation points needed for a linear operator to reach a required accuracy \cite{Wozniakowski1994}. 

Wozniakowski stated in \cite{Wozniakowski1994} that for a given $\epsilon$ and $d$, the minimal number of interpolation points $n = n(\epsilon, d)$, for a linear interpolation operator $L_n(f) = \sum_{j=1}^n f(x_j)c_j$ to approximate a function $f$ that satisfies $\|f\|_k \le 1$ in the $L_2$ norm with precision $\epsilon$, is bounded by 
\begin{equation}
    \label{eq:sparse_grid}
    n(\epsilon, d) \ge c_\epsilon d^{c \log(\epsilon^{-1})}
\end{equation}
where $c_j \in C([-1,1]^d)$, and $\|f\|_k^2 := \sum_{l \in \mathbb{N}_0}(1+l^2)^k a_l^2[f]$ with $a_l[f]$ denoting the Fourier coefficient of $f$.

We establish that the \frank{} in \Cref{thm:cheb1} is equivalent to $n(\epsilon, d)$ described above. We start with the assumptions. 
In \Cref{thm:cheb1}, the analytic assumption implies that $\|f\|_k \le 1$, and the $L_\infty$-norm error suggests the same results hold for $L_2$-norm error if we assume the volume of the domain is bounded by 1. We then connect the number of points from a function interpolation to the number of terms from a function expansion by the following formula:
\begin{equation}
\label{eq:interpolation_expansion}
        \calK(\vecx, \vecy) = \sum_{i=1}^{n} \calK (\vecx_i, \vecy) c_i(\vecx) + \epsilon
    \end{equation}
    Therefore, we have established the equivalence of the \frank{} in \Cref{thm:cheb1} and $n(\epsilon, d)$,
    and we conclude that our \frank{} reaches the lower bound in \cref{eq:sparse_grid} asymptotically.

    {\bf Related work.}
Barthelmann \cite{Barthelmann2000} considered a polynomial interpolation on a sparse grid, and showed that such interpolation could reach an acceptable accuracy with the number of interpolation points growing polynomially with the data dimension.
Specifically, consider a real-valued function $f$ defined on $[-1,1]^d$ with its derivative $D^{\vecalpha} f$ being continuous for $\|\vecalpha\|_\infty \le k$.
    If we interpolate $f$ using the Smolyak formula \cite{Smolyak1963}, then the interpolation error in the 0-norm is bounded by
\begin{equation}
    \label{eq:sparse2}
    c_{d,k} N^{-k} (\log N)^{(k+1)(d-1)}  \|f\|_k
\end{equation}
     where the norms $\|\cdot\|_0$ and $\|\cdot\|_k$ adopt the same notations as above.  The number of interpolation points used (see \cite{Novak1999}) is
\begin{equation}
    N = N(k+d, d) = \sum_{s = 0}^{\min (k,d)} \binom{k}{s}  \binom{k+d-s}{k} \le \binom{2k+1+d}{d}
\end{equation}

Consider $N(k+d, d)$. When $k$ is fixed, and $d \rightarrow \infty$, the number of points used in the Smolyak technique roughly behaves as $O(d^k)$. We use the same argument between the lines of \cref{eq:interpolation_expansion} to connect the number of function interpolation points and the number of function expansion terms, and conclude that the polynomial dependence on $d$ is consistent with our result in \cref{eq:kernel_seperate_form}.

In the following section, we use the matrix rank to verify our
theoretical results on the function rank. We have mentioned in
\Cref{sec:intro} that the \frank{} is an upper bound of the \mrank{}. Hence, we would expect the matrix rank related to the max norm to grow polynomially with $d$ as well. The low-rank representation of a kernel function and its approximation error can be related to those of a kernel matrix defined on the same domain in the following way. If a kernel function $\calK$ can be approximated by the separable form $\sum_{i=1}^R g_i(\vecx)h_i(\vecy)$ with $L_\infty$ error $\epsilon$, then for an $n$ by $n$ kernel matrix $K$ with entries $K_{ij} = \calK(\vecx_i, \vecy_j)$, it is straightforward to construct a low-rank representation $GH^T$ of $K$ with rank at most $R$, where $G_{ij} = g_j(\vecx_i)$ and $H_{ij} = h_j(\vecy_i)$. And, the matrix approximation error in the Frobenius-, two-, and max- norm is bounded by $\epsilon n$, $\epsilon n$, and $\epsilon$, respectively. 

Now that the connections between \mrank{} and \frank{} have been established explicitly, we can move on to the numerical experiments. 

\section{Numerical experiments}
\label{sec:experiments}
    In this section, we experimentally verify two main results from our
    theorems: the polynomial growth of the numerical \frank{} with the
    data dimension, and the influence of the diameters of $\Omega_\vecx$
    and $\Omega_\vecy$ on the approximation error. By the arguments
    before the beginning of this section, we will use the \mrank{} to verify the behavior on the
    \frank{}. We report the \mrank{} for various data distributions due to our worst-case error bounds.

     \subsection{Experimental settings}
     We consider first the data distribution in the experiments. Generating data which is representative of the worst case is
       difficult. On the one hand, sampling randomly from
       common distributions will cause the empirical variance of the pair-wise distances to decrease with $d$, due to concentration of measure; on the other hand, designing the points to achieve a large empirical variance will require correlations among points and cause them to lie on a
       manifold. Both methods will yield matrices with lower
       \mrank{}s. Considering that the {RBF}s are functions of the distances, we seek distributions of points in a unit cube of dimension $d$ such that the pair-wise distances follow a probability distribution whose variance decreases slower with $d$, and the points do not lie approximately on a manifold of the domain.
       
       For a limited number of points that is imposed by the computational limit and for large $d$, a fast decay of the empirical variance is observed for quasi-uniform distributions of points, \emph{e.g.}, using data generated from perturbed grid points or Halton points. The pair-wise distances of Halton points and uniform sampled points fell into a small-sized subinterval of $[0, \sqrt{d}]$ that is away from the endpoint $\sqrt{d}$, reducing the range of observed distances, leading to spurious low-ranks.
               
        We propose a sampling distribution---which we call the endpoint distribution---to encourage the occurrence of large distances that would otherwise not be covered with a high probability. Specifically, for a random variable $X$, $\text{Pr}(X = a) = \text{Pr}(X = b) = p_d$, $\text{Pr}(a < X < b) = 1-2p_d$, $\text{Pr}(X < a) = \text{Pr}(X > b) = 0$, where $p_d$ was selected by a grid search to yield the largest rank for each $d$. The range of the covered domain is much wider than either using Halton points or uniform sampling. 
        
    We consider next the numerical \mrank{} that will be reported in the results. The numerical \mrank{} associated with tolerance $tol$ is
        $$R_{tol} = \min \left\{ r \mid \|K - U_rS_rV_r^T\| \le tol \, \|K\| \right\},$$
    where $U_r, S_r, V_r$ are factors from the singular value
    decomposition ({SVD}) of the matrix $K$. Depending on the choice of the
    norm, the value of $R_{tol}$ will vary. Our main focus is on the
    max norm, which is consistent with the function infinity norm in the
    theorems. Theoretically, the max error does not decrease
    monotonically with the \mrank{}; however, we found that for
    the RBF kernel matrices, the max error decreases in general with the
    \mrank{}, except for certain small, short-lived increases.

    Throughout our experiments, we fix the number of points at 10,000. The kernel used is the Gaussian kernel $\exp(-\|\vecx - \vecy\|_2^2/h^2)$ with $h = \sqrt{d}$. The data were generated from the above endpoint distribution with endpoints to be 0 and 1. For each set of dimension and tolerance, we report the mean and standard deviation of the numerical matrix rank out of 5 independent runs. 
    
     \subsection{Experimental results}
    \Cref{fig:rank_vs_dim} shows the numerical \mrank{} as a function of data dimension subject to a fixed tolerance on 3 different data overlapping scenarios: source and target data both in $[0,1]^d$; source data in $[0, 2/3]^d$ and target data in $[1/3, 1]^d$; and source data in $[0, 1/2]^d$ and target data in $[1/2, 1]^d$. By design, the ratio between $D_\vecx \,(\text{or}~D_\vecy)$ of these scenarios is roughly 6 : 4 : 3 and they are shown from top to bottom for each fixed tolerance in \Cref{fig:rank_vs_dim}.

    \begin{figure}[htbp]
        \centering
            \begin{subfigure}[b]{0.32\textwidth}
                \includegraphics[width=\textwidth]{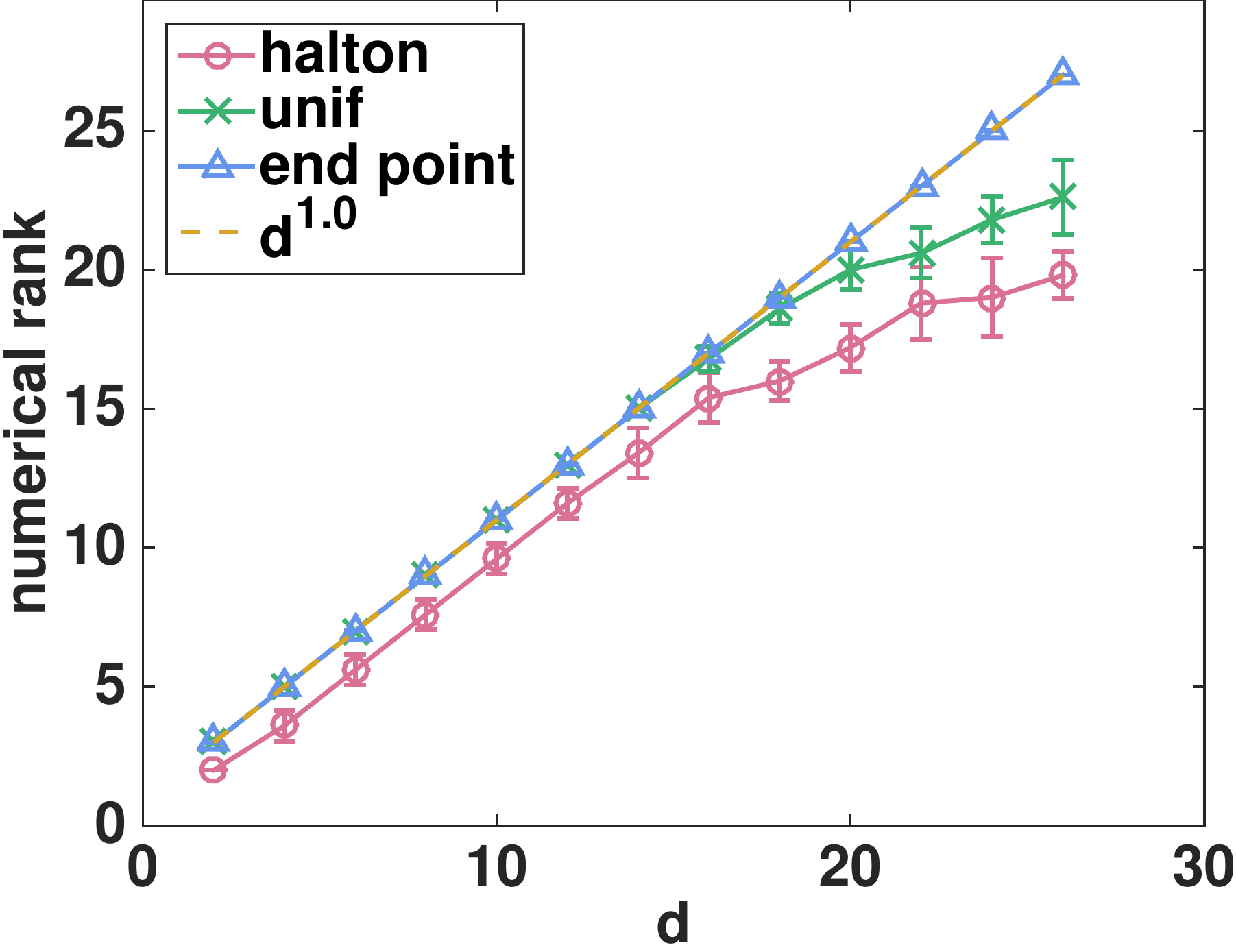}
                \caption{tol = $10^{-1}$}
            \end{subfigure}
            \begin{subfigure}[b]{0.32\textwidth}
                \includegraphics[width=\textwidth]{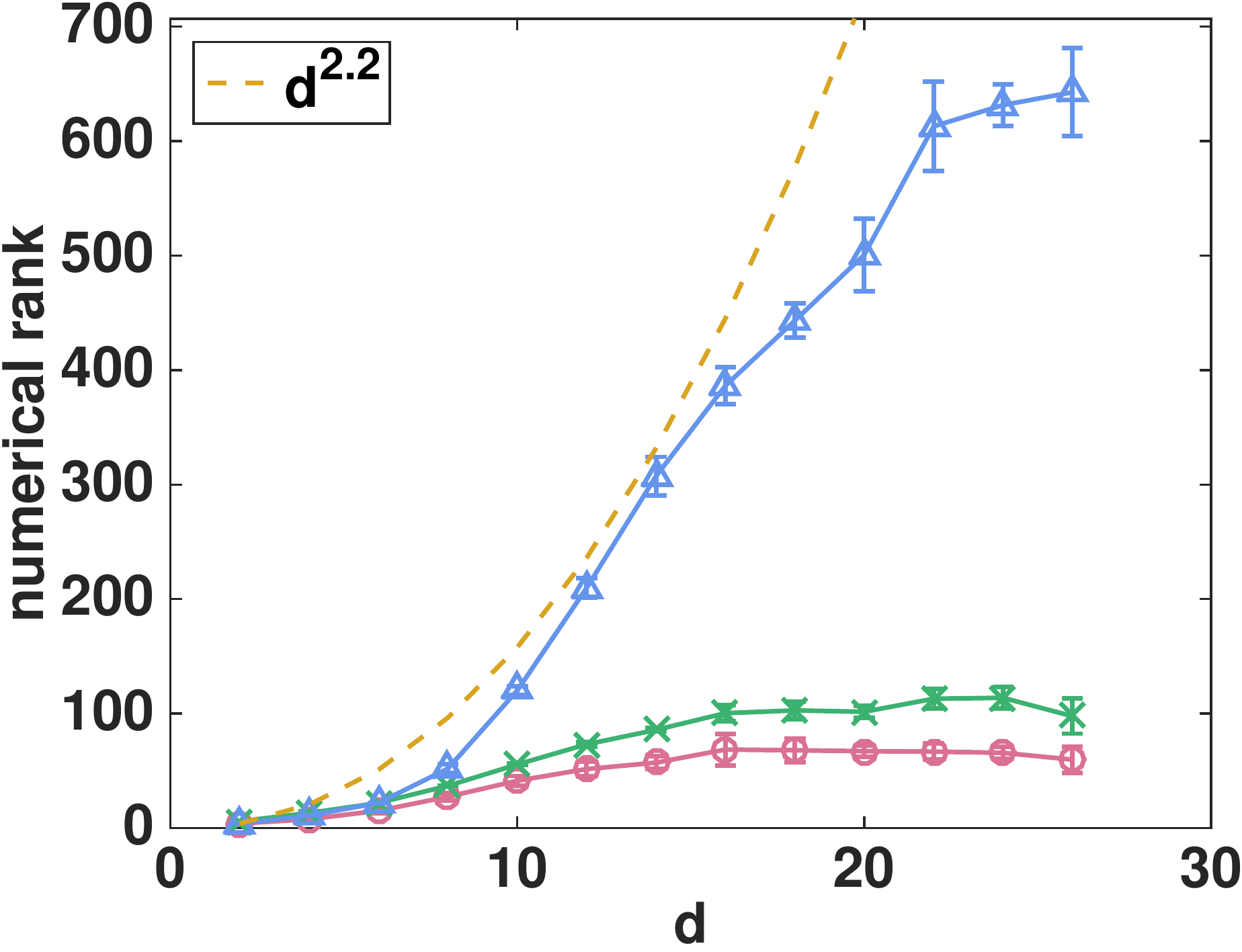}
                \caption{tol = $10^{-2}$}
            \end{subfigure}
            \begin{subfigure}[b]{0.32\textwidth}
                \includegraphics[width=\textwidth]{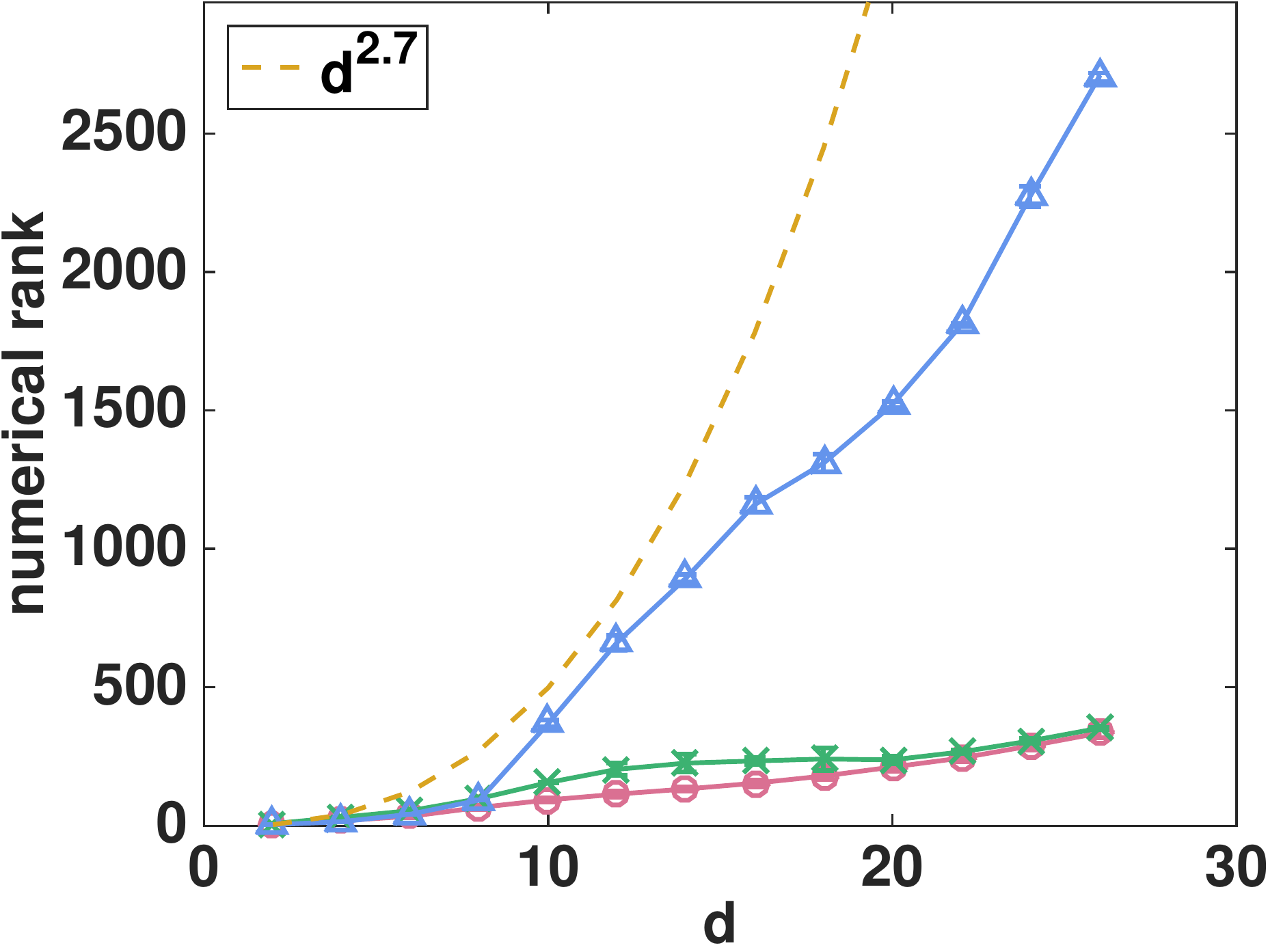}
                \caption{tol = $10^{-3}$}
            \end{subfigure}
            \vspace{-.1in}
            {\caption*{\bf completely overlapped}}
                 \begin{subfigure}[b]{0.32\textwidth}
                     \includegraphics[width=\textwidth]{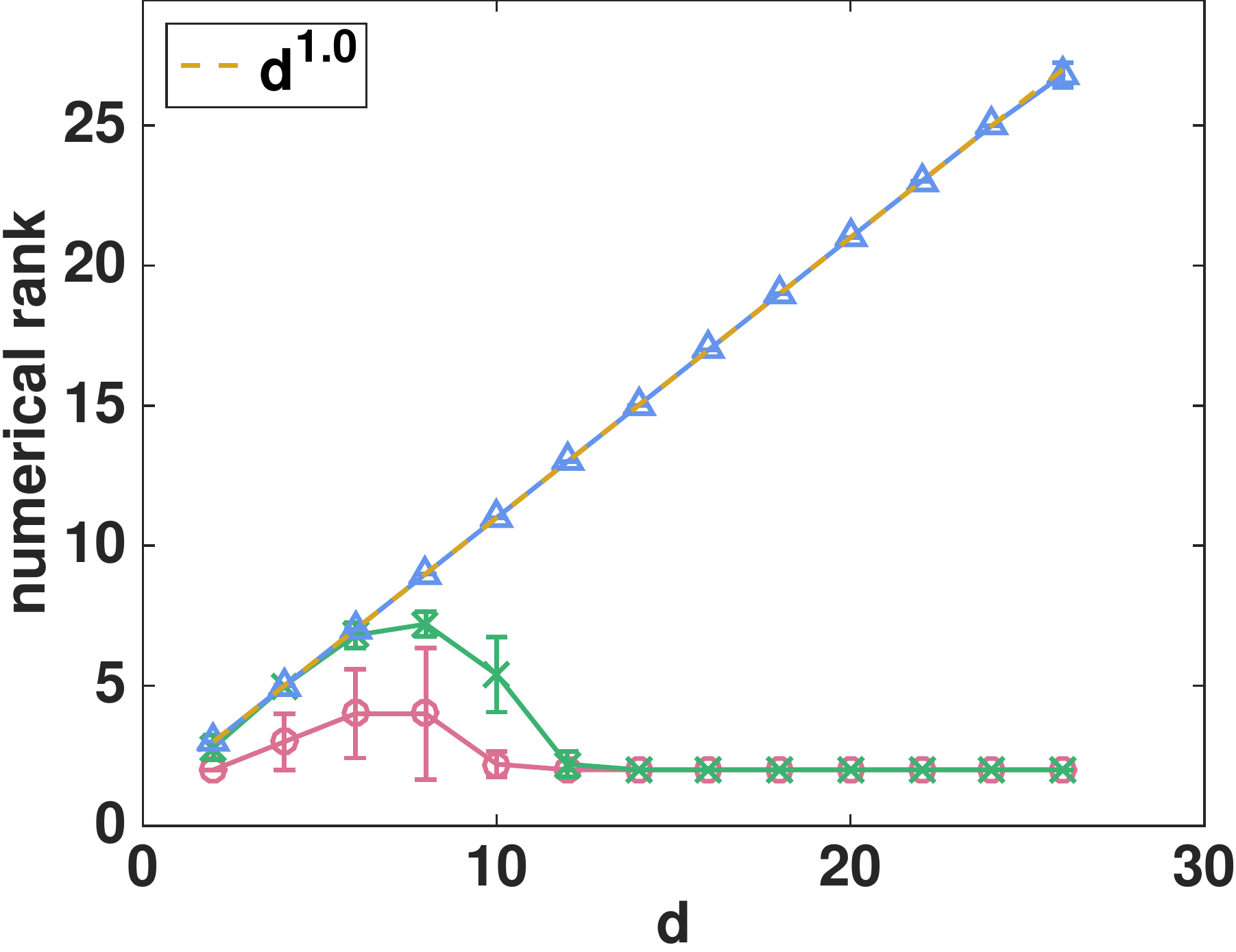}
                     \caption{tol = $10^{-1}$}
                 \end{subfigure}
            \begin{subfigure}[b]{0.32\textwidth}
                \includegraphics[width=\textwidth]{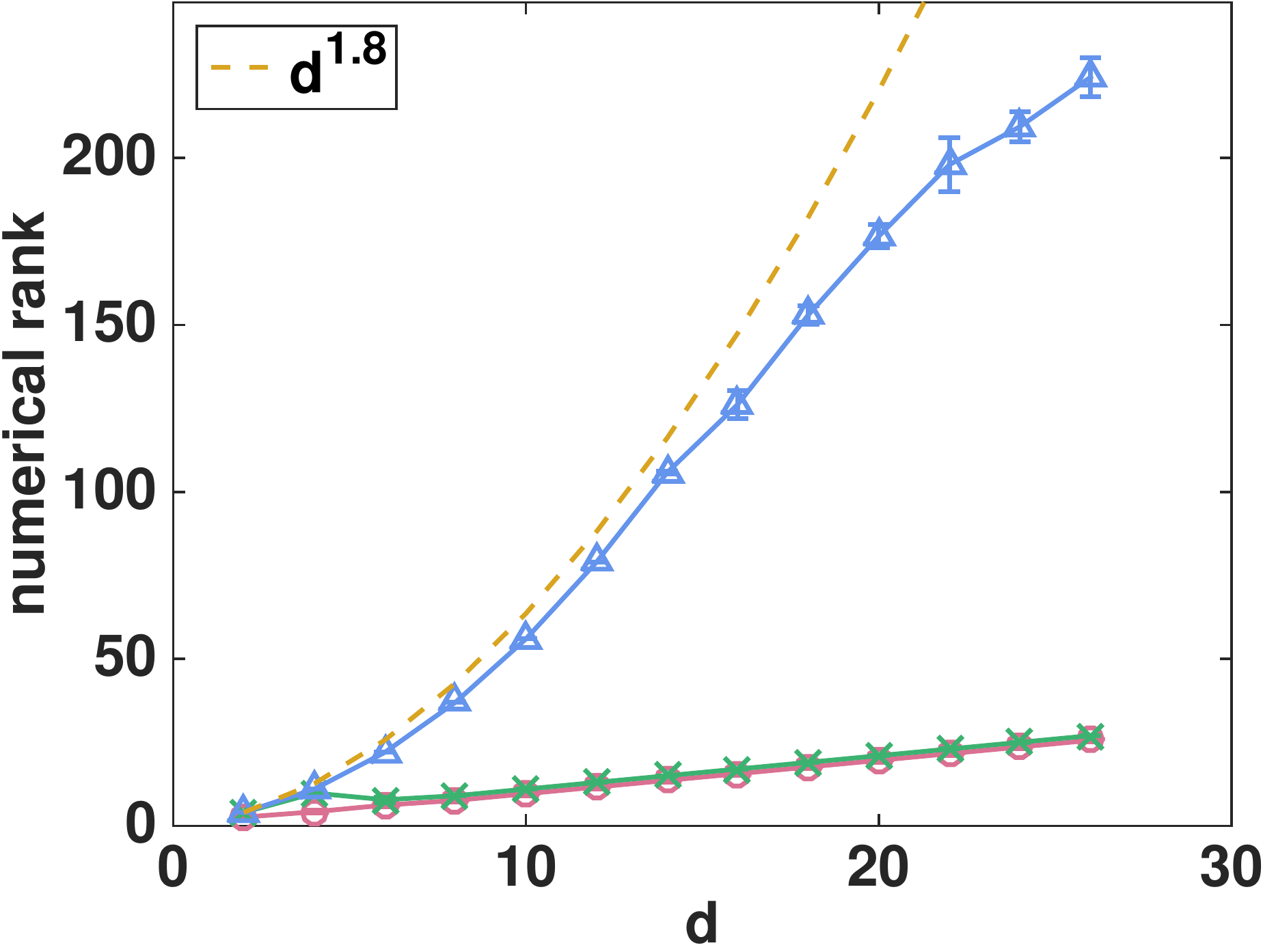}
                \caption{tol = $10^{-2}$}
            \end{subfigure}
            \begin{subfigure}[b]{0.32\textwidth}
                \includegraphics[width=\textwidth]{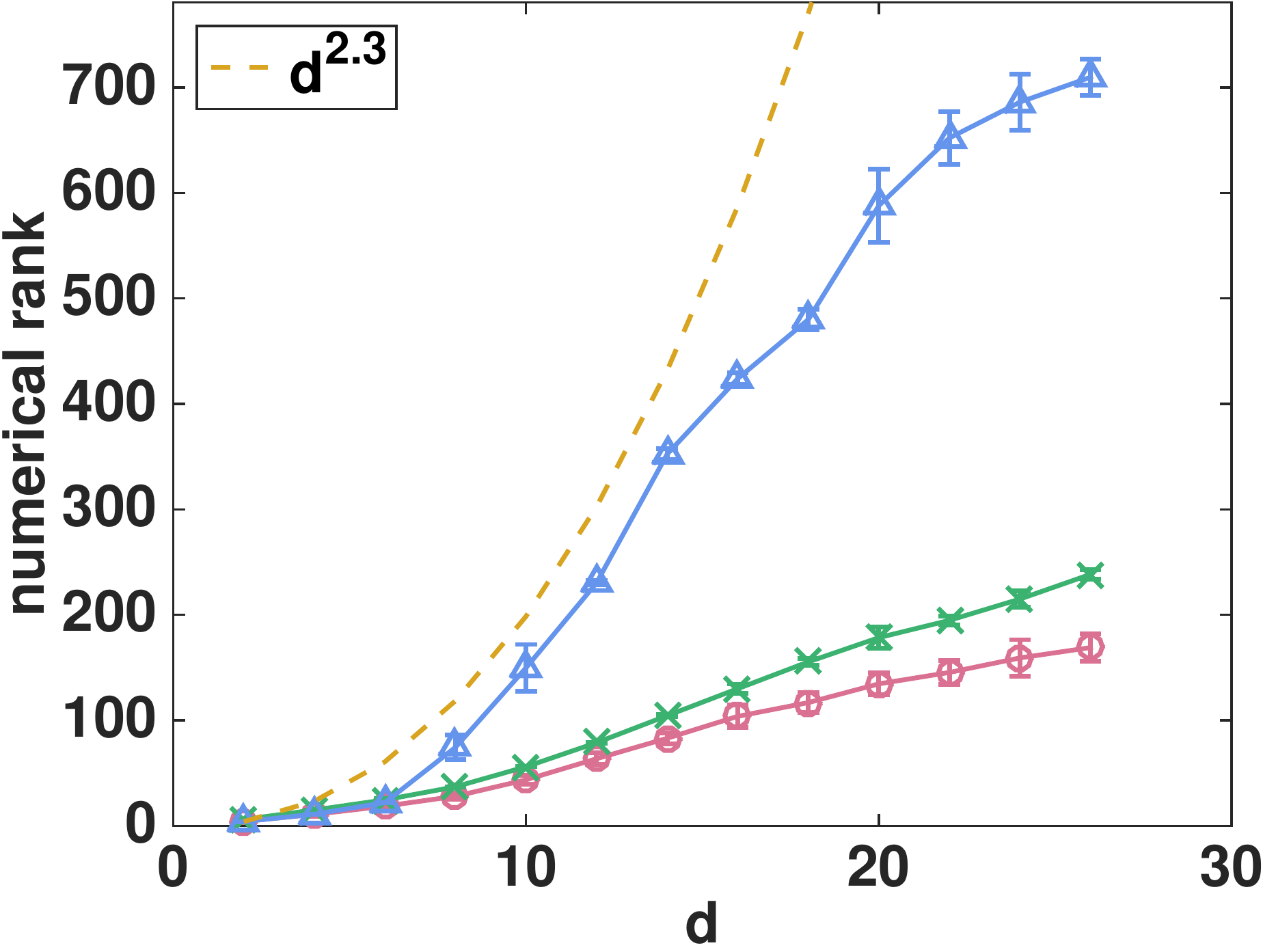}
                \caption{tol = $10^{-3}$}
            \end{subfigure}
            \vspace{-.1in}
            {\caption*{\bf partially overlapped}}
                \begin{subfigure}[b]{0.32\textwidth}
                    \includegraphics[width=\textwidth]{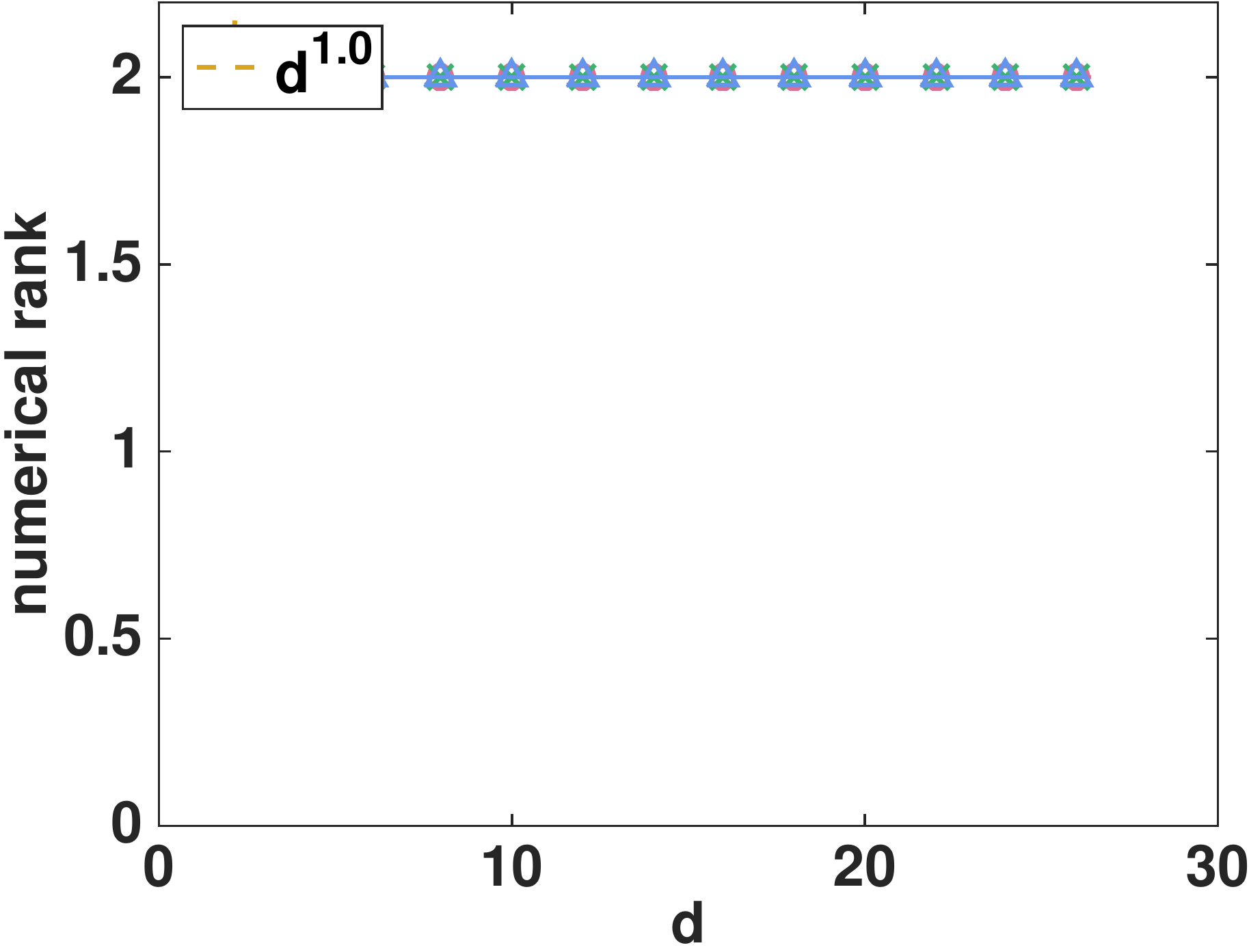}
                    \caption{tol = $10^{-1}$}
                \end{subfigure}
            \begin{subfigure}[b]{0.32\textwidth}
                \includegraphics[width=\textwidth]{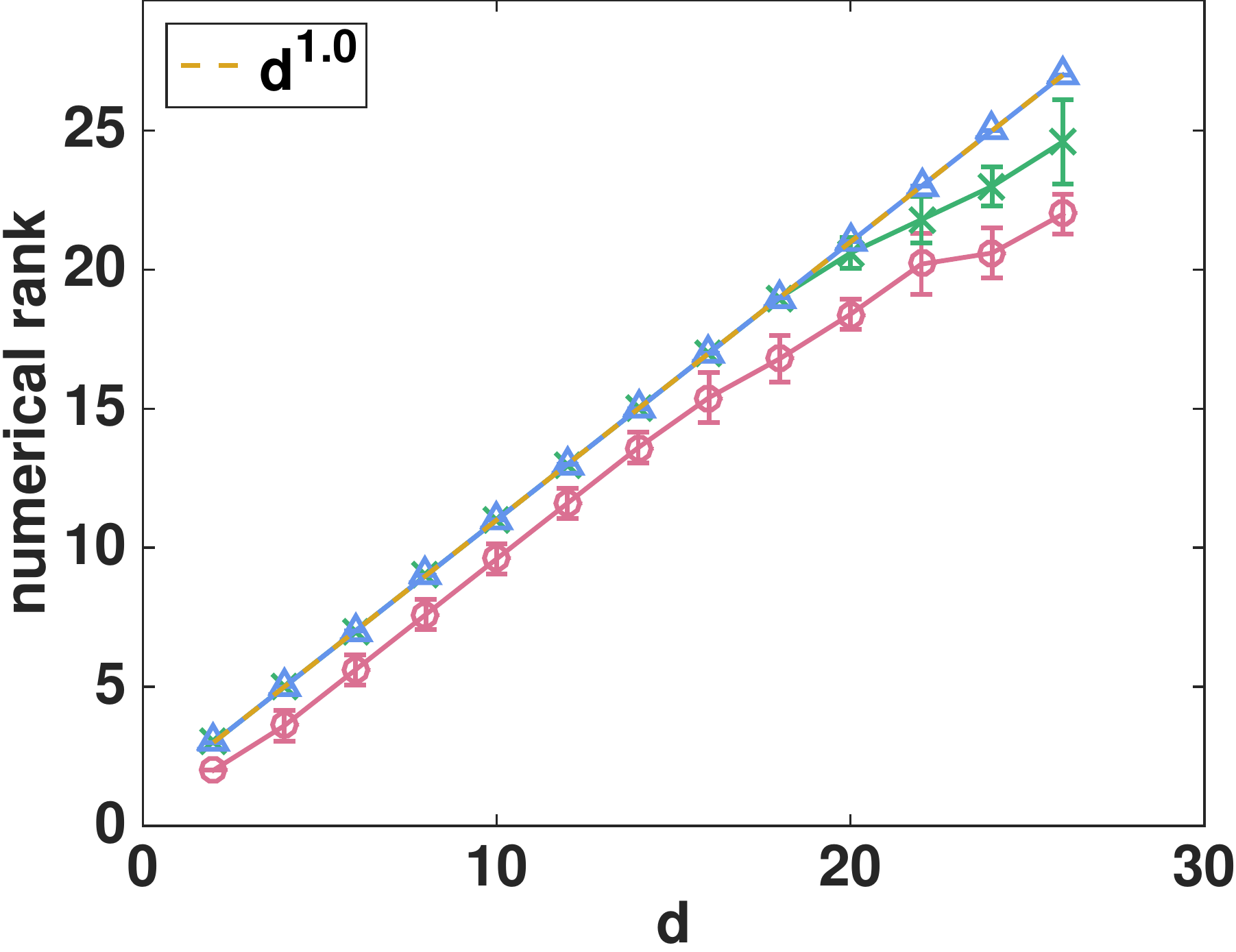}
                \caption{tol = $10^{-2}$}
            \end{subfigure}
            \begin{subfigure}[b]{0.32\textwidth}
                \includegraphics[width=\textwidth]{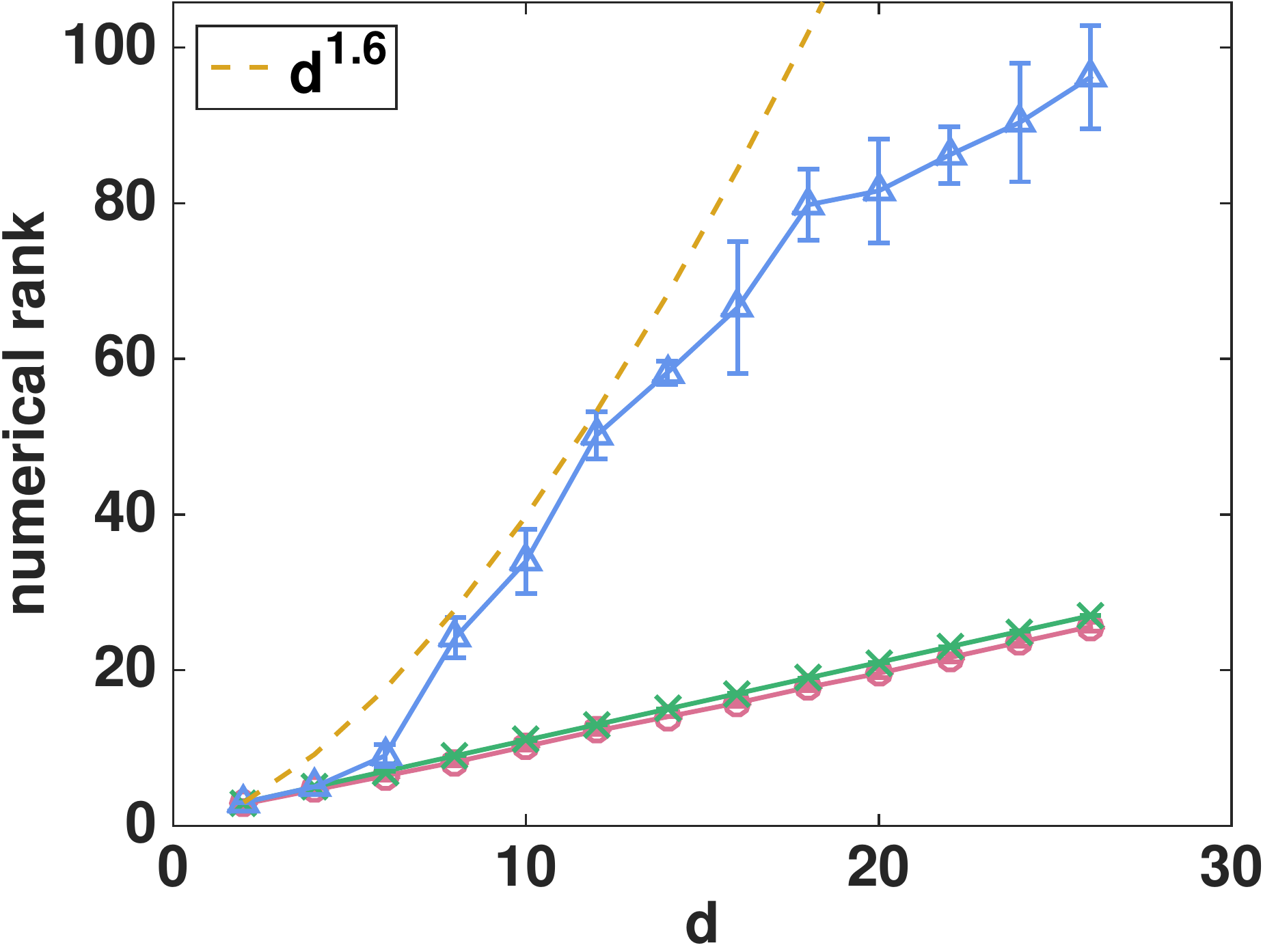}
                \caption{tol = $10^{-3}$}
            \end{subfigure}
            \vspace{-.1in}
            {\caption*{\bf not overlapped}}
        \caption{Numerical rank vs.\ data dimension with different sampling methods. The rank was related to the max norm, and the data size was fixed at 10,000. Subplots shared the same legend, where ``halton" is Halton set; ``unif" is uniform sampling; and ``end point" is our proposed sampling. Subplots considered different data scenarios, in which the regions containing the source and target points either completely overlap [$(a)$ to $(c)$], partially overlap [$(d)$ to $(f)$], or do not overlap [$(g)$ to $(i)$].}
        \label{fig:rank_vs_dim}
    \end{figure}

    The plots along each row verify that for a fixed $n$ that
    represents the polynomial order in the low-rank representation, the
    \frank{} grows as $O(d^n)$ with $d$. In our experiments, we increase $n$ by decreasing the approximation tolerance, according to the relation between order $n$ and error $\epsilon$ in \Cref{thm:cheb1}. We observe results consistent with the order $O(d^n)$.

    The plots along each column verify that decreasing the domain diameter for either $\Omega_\vecx$ or $\Omega_\vecy$ reduces the error bound. \Cref{thm:fourier_taylor} suggests that $D_\vecx$ and $D_\vecy$ influence the error in the form of $\left(\frac{D_\vecx D_\vecy}{D^2}\right)^{M_t + 1}$. That is, to maintain a certain precision, a smaller domain diameter allows $M_t$ to be smaller, and consequently allows the rank to be smaller. This relation of domain diameter and error bound is verified by our experimental results when observing from top to bottom. 

    \Cref{fig:rank_vs_dim_varying_norms} further reports the \mrank{} related to different norms. In particular, the \mrank{} related to the Frobenius-norm and the two-norm increases with $d$ in the small-$d$ regime, and in the large-$d$ regime, it decreases. This is an interesting observation. Regretfully, we cannot provide a clear explanation based on our theorems; we will only describe our observation in the paper and leave the theory for future work. 
    
        \begin{figure}[htbp]
        \centering
            \begin{subfigure}[b]{0.32\textwidth}  
 \includegraphics[width=\textwidth]{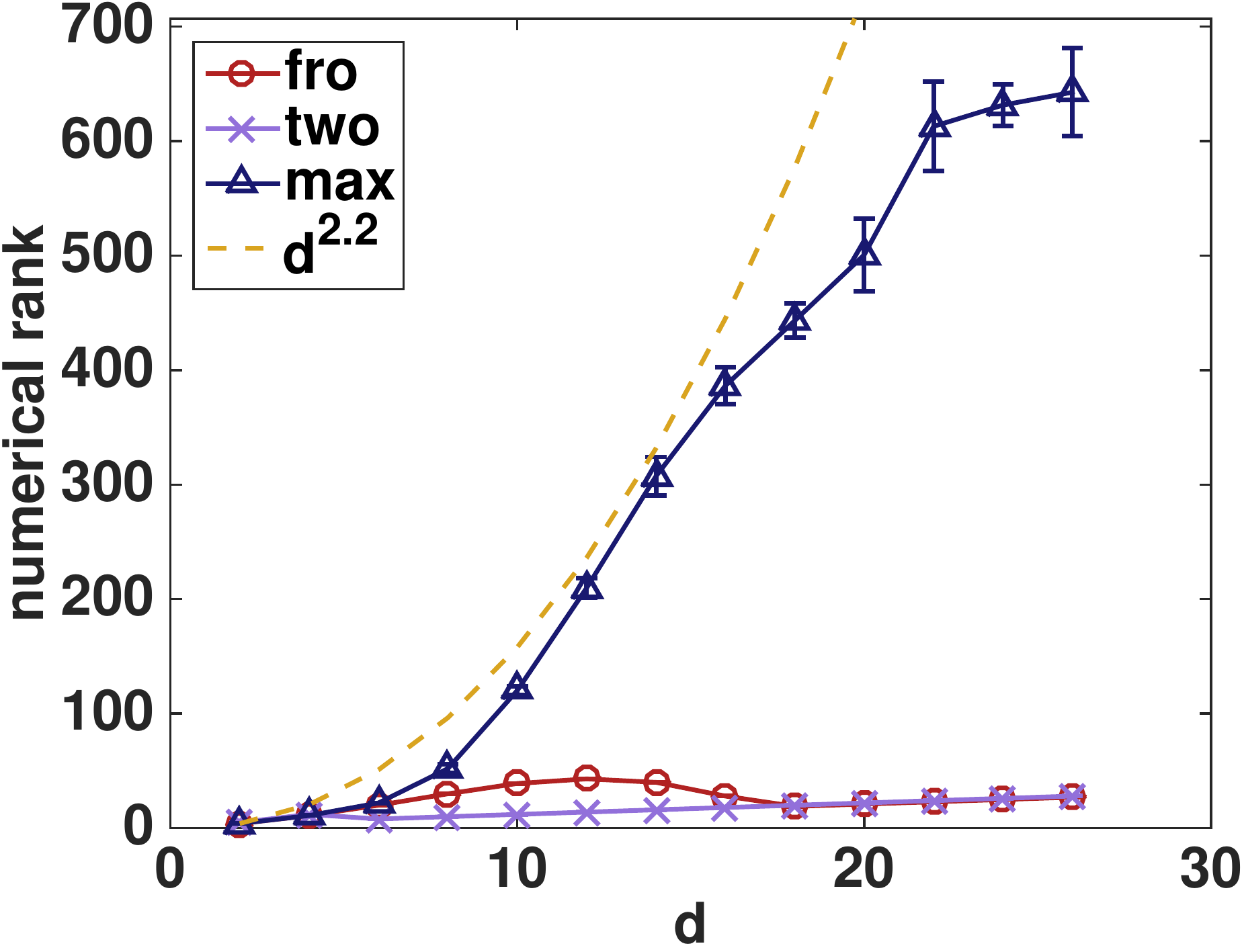}
                \caption{complete overlapped}
            \end{subfigure}
            \begin{subfigure}[b]{0.32\textwidth}  
 \includegraphics[width=\textwidth]{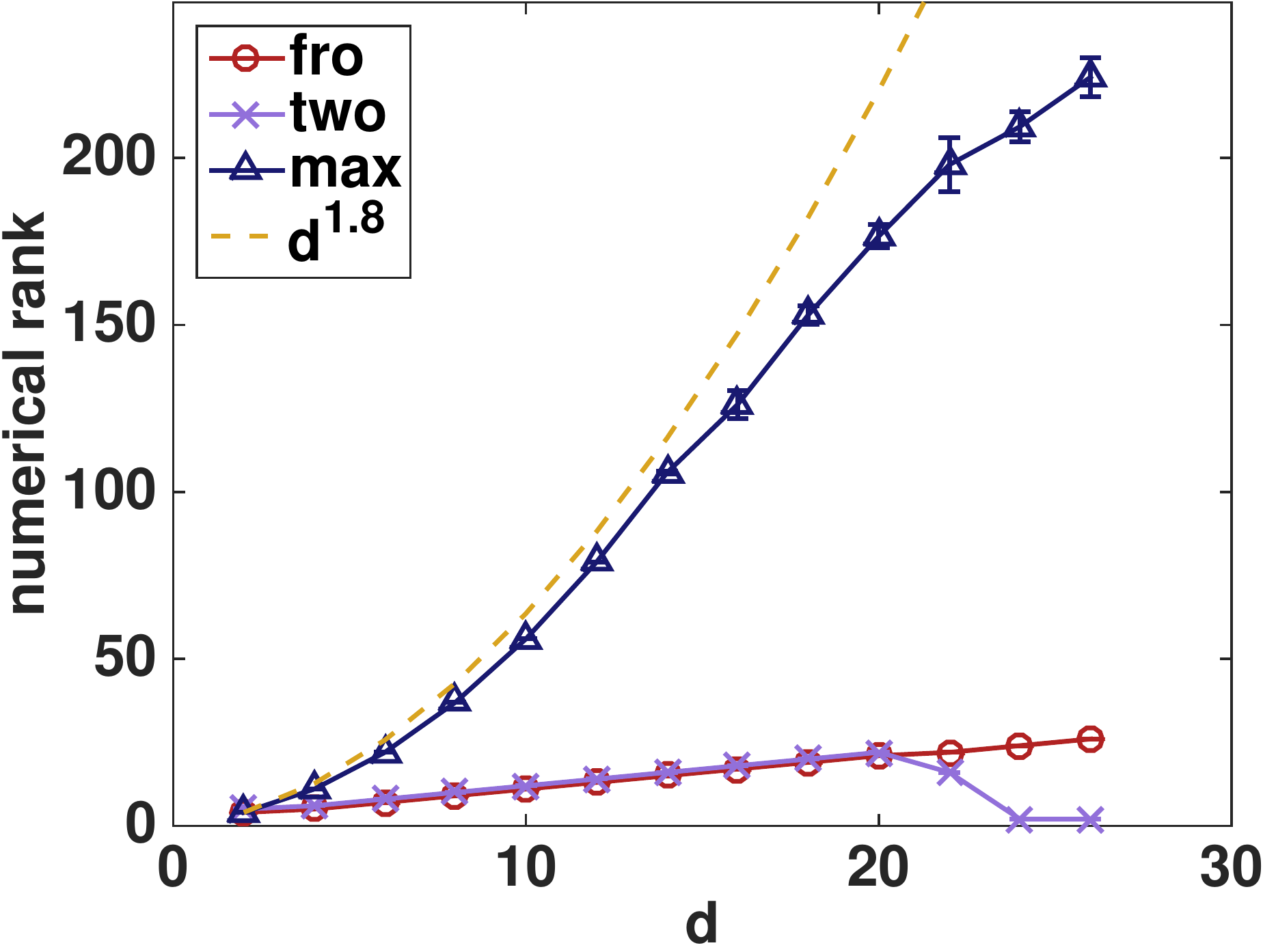}
                \caption{partial overlapped}
            \end{subfigure}
        	\begin{subfigure}[b]{0.32\textwidth}    \includegraphics[width=\textwidth]{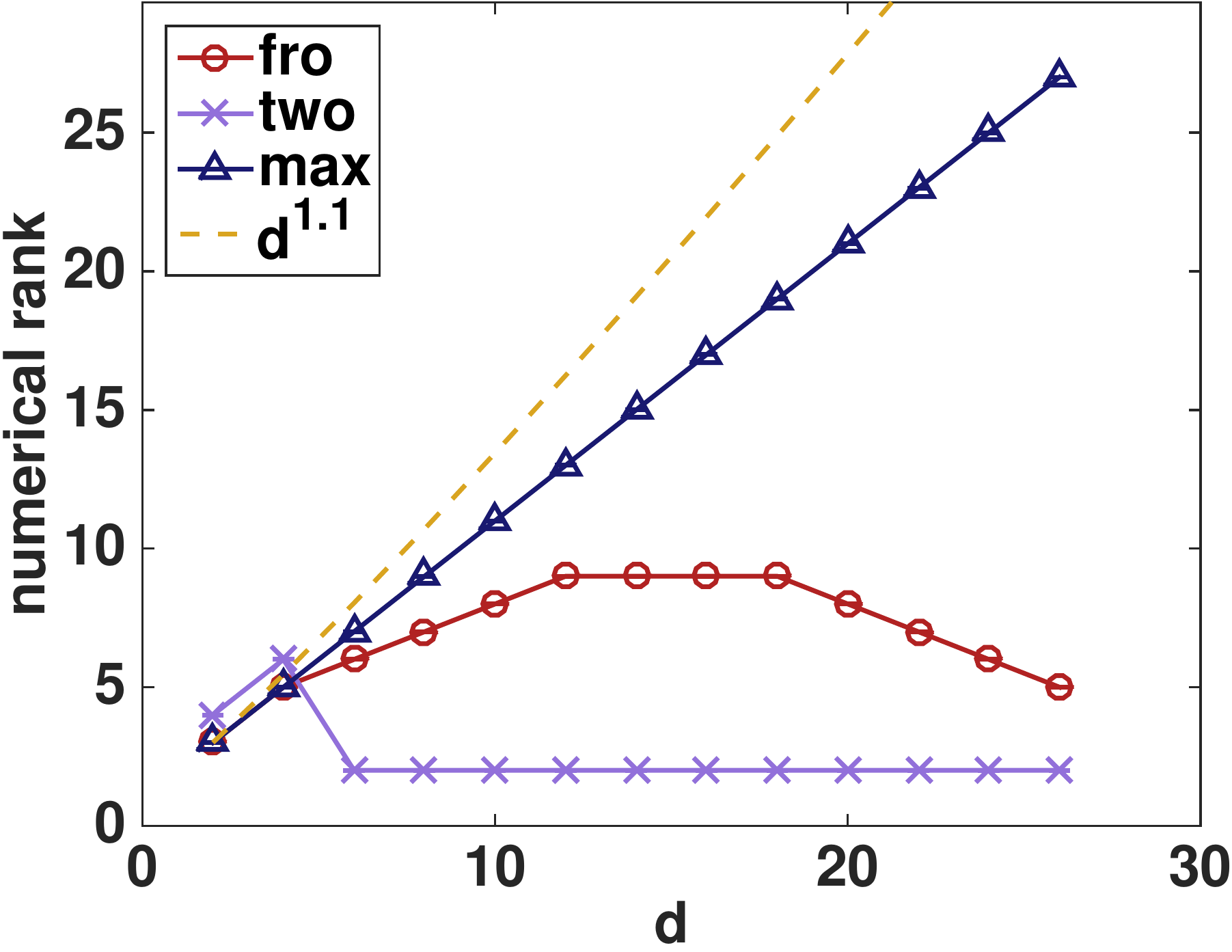}
                \caption{not overlapped}
            \end{subfigure} 
        \caption{Numerical rank vs.\ data dimension with rank related to different norms. The data size was fixed at 10,000 and the data were sampled from our end point distribution. The legend lists the choice of norms in the rank definition $\min \{ r \mid \|K - U_rS_rV_r^T\| \le tol \, \|K\| \}$, where ``fro" is Frobenius norm; ``two" is two norm; and ``max" is max norm. The tolerance was fixed at $10^{-2}$.}
        \label{fig:rank_vs_dim_varying_norms}
    \end{figure}
 
 	 To summarize, up to some precision, smooth {RBF} kernels behave
     like kernels constructed by summations of products of functions of
     $\vecx$ and of $\vecy$. For a fixed kernel on a fixed domain, the
     maximal total degree of those products and the dimension altogether
     determine the observed \frank{} in practice. And, the dimension
     influence on the \frank{} is only a power of the dimension, and the
     power depends on the accuracy. In addition, this is still the
     worst-case scenario, attained for large and regular point sets. The
     real-world data are often more structured and rarely realize the
     worst-case, and for a fixed kernel and the practical data, the
     low-rank approximations would have lower \frank{}s, and hence the
     corresponding kernel matrices would have lower \mrank{}s.

\section{Group pattern of singular values}
\label{sec:plateau_decay}
In this section, we reveal and explain a group pattern in the singular values of RBF kernel matrices. Specifically, the singular values form groups by their magnitudes, and the group cardinalities are dependent on the data dimension and independent of the data size.

	If we order the singular values from large to small, then the indices where significant decays occur can be described as $\binom{k+d}{d}$. This number is a cumulative sum of the dimensions of the $d$-variate polynomial spaces arising in the terms of the truncated power series kernel $\sum_{|\vecalpha| \le k} c_{\vecalpha} \vecx^{\vecalpha} \vecy^{\vecalpha}$ up to order $k$, which is close to our separable form of the kernel in a loose sense. 
    
    However, this formula fails to capture those less significant decays. We, therefore, explain the group pattern based on \Cref{thm:fourier_taylor} by an appropriate grouping of the number of terms in the function's separable form. For any RBF, consider the number of separate terms $n(M_f, M_{j})$ in its separable form:
   \begin{equation}
       \begin{split}
           \label{eq:decay_terms}
           n(M_f, M_{j}) = \sum_{j=0}^{M_f} \sum_{k = 0}^{M_{j}} n_k = \sum_{j=0}^{M_f} \sum_{k = 0}^{M_{j}} \binom{k+d-1}{k} = \sum_{j=0}^{M_f} \binom{M_{j}+d}{d}
       \end{split}
   \end{equation}
The two summations correspond to the Fourier expansion of the kernel function, and the Taylor expansion of each Fourier term, respectively. Let $n_k$ denote the number of separate terms in $\left( \vecrho_{\vecx}^T \vecrho_{\vecy} \right)^k$ that occurs in the $k$-th order Taylor term. The observed group cardinalities are described by a grouping of the terms in \cref{eq:decay_terms}, whose order is governed by the truncation error. One grouping example is
\[
    \underbrace{n_0, n_1,n_2}_{\text{1st term of Fourier expansion}} \mid \underbrace{n_0, n_1}_{\text{2nd term of Fourier expansion}} \mid \underbrace{n_3, n_4}_{\text{1st term of Fourier expansion}}
\]
    The cardinality of the 1st, 2nd and 3rd group is $n_0+n_1+n_2$, $n_0+n_1$ and $n_3+n_4$, respectively, and a cumulative sum of which yields the decay indices. The formula given by \cref{eq:decay_terms} generalizes that given by the dimension of the polynomial space. In the special case where only the first-order Fourier term is considered, these two formulas agree, namely the number of the Taylor terms up to order $k$ matches the dimension of the $d$-variate polynomial space of maximum degree $k$. 

    \subsection{Experimental Verification}
We experimentally verify the above claim.

    \Cref{fig:sval_ratio} shows the ratio of
    the $i$-th largest singular value to the next smaller one. We are interested in the group cardinality and the singular value decay amount, which altogether determine the \mrank{}. The group
    cardinality is the distance between two adjacent high-ratio indices, and the singular value decay amount
    is indicated by the magnitudes of the ratio (the spike). These two quantities are independent of the data size, suggesting that for a fixed kernel and a fixed precision, the numerical \mrank{} is independent of the data size, assuming the data does not lie in a manifold. Additionally, this also verifies an earlier statement that as the point sets in a fixed domain become denser, the rank and the error remain unchanged. 
    \begin{figure}[htbp]
        \centering
            \begin{subfigure}[b]{0.45\textwidth}
                \includegraphics[width=\textwidth]{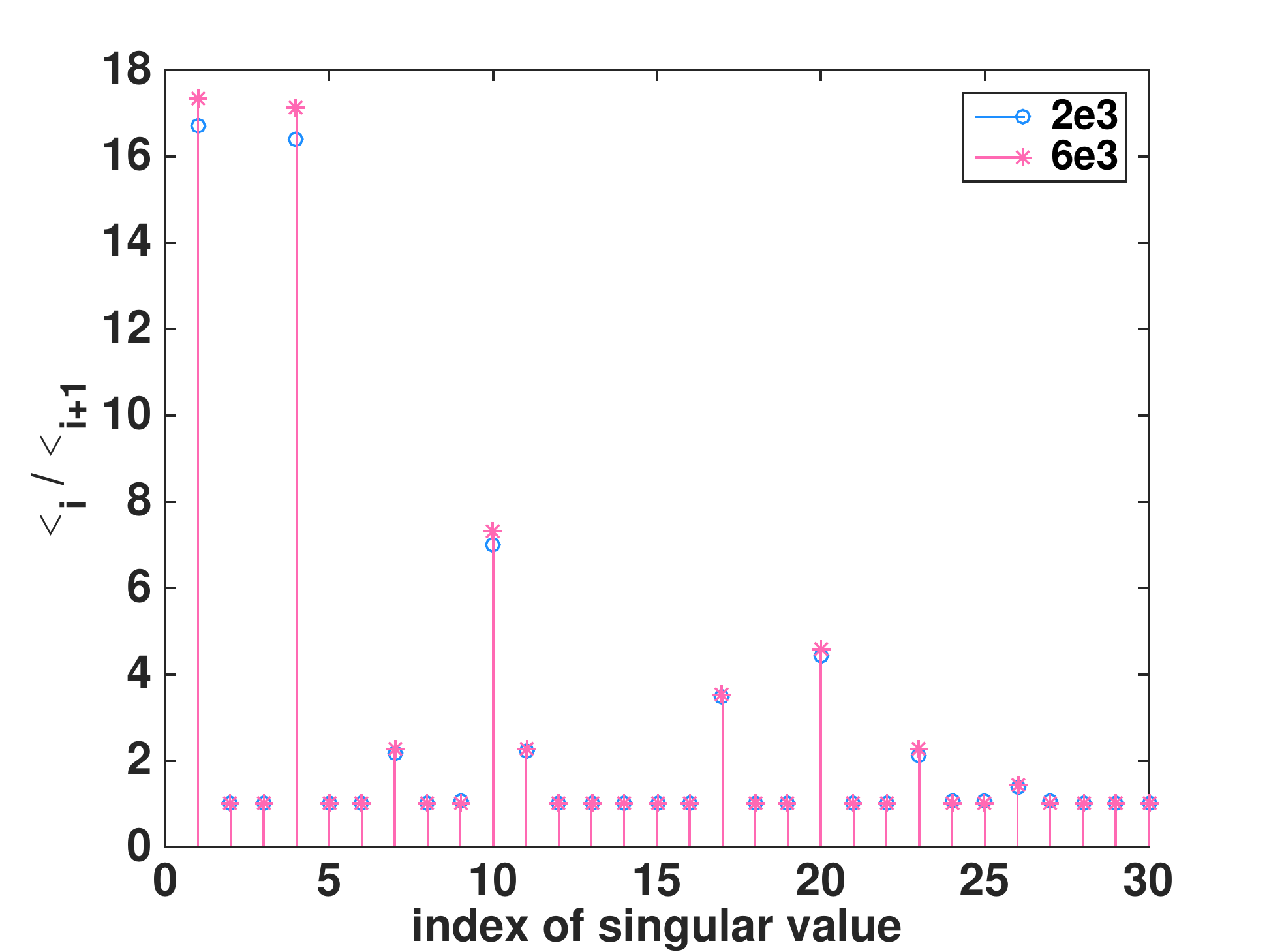}
                \caption{$d = 3$, $\exp(-\norm{\vecx-\vecy}_2^2/h^2)$}
                \label{fig:sval_ratio_d3_gaussian}
            \end{subfigure}
            \begin{subfigure}[b]{0.45\textwidth}
                \includegraphics[width=\textwidth]{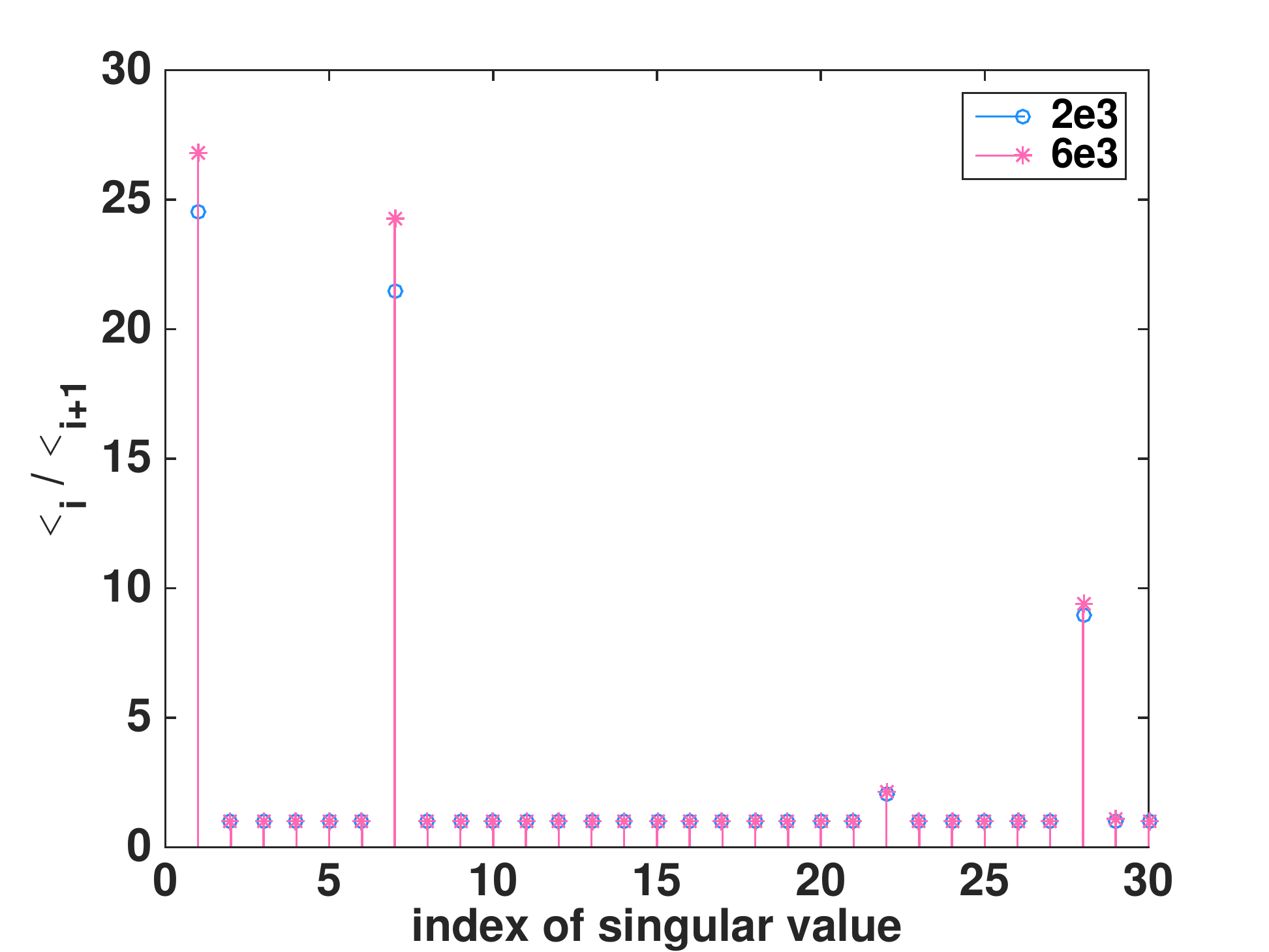}
                \caption{$d = 6$, $\exp(-\norm{\vecx-\vecy}_2^2/h^2)$}
                \label{fig:sval_ratio_d6}
            \end{subfigure}
            \begin{subfigure}[b]{0.45\textwidth}
                \includegraphics[width=\textwidth]{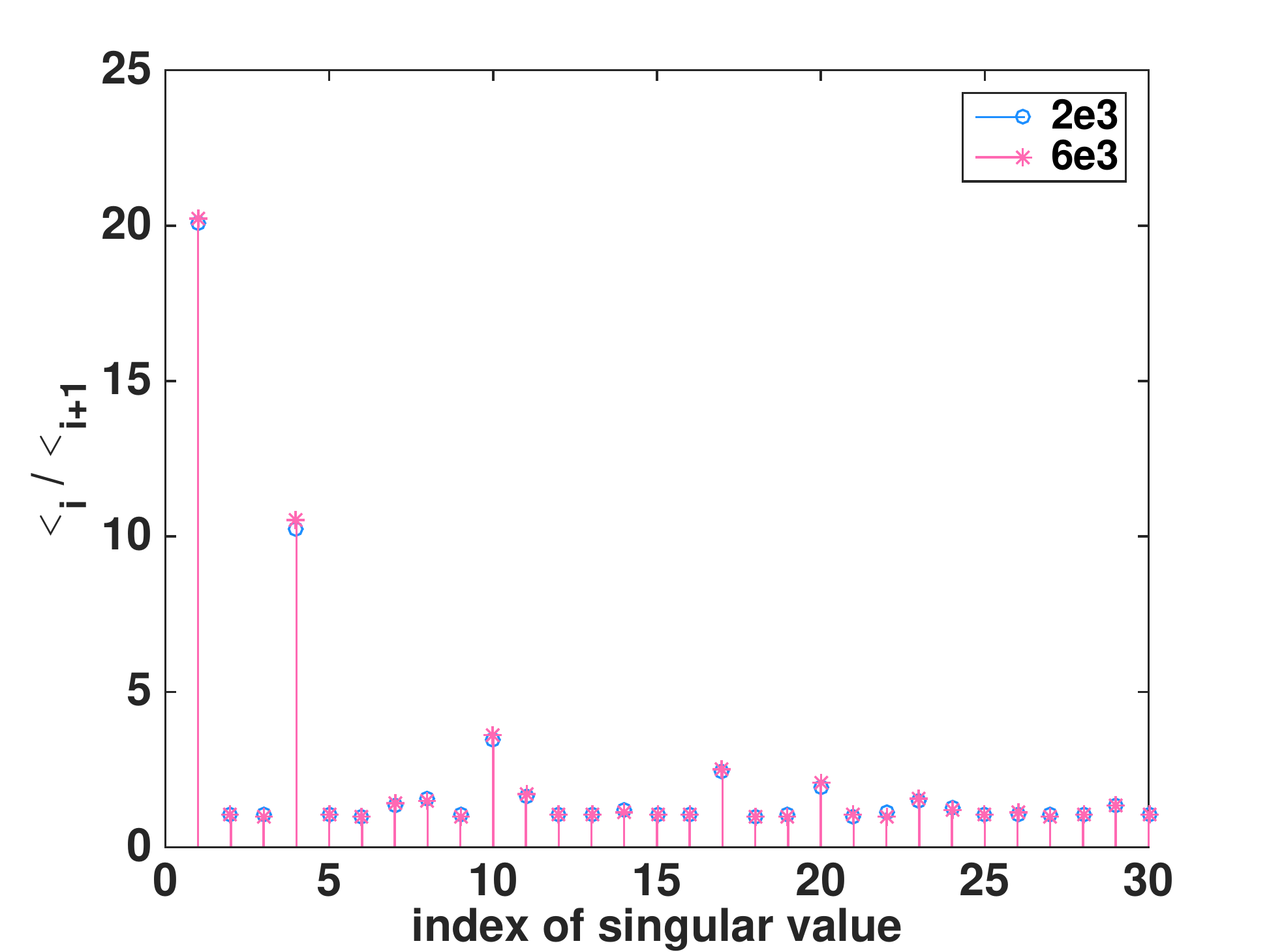}
                \caption{$d = 3$, $\frac{1}{1+\norm{\vecx-\vecy}_2^2/h^2}$}
                \label{fig:sval_ratio_d3}
            \end{subfigure}
            \begin{subfigure}[b]{0.45\textwidth}
                \includegraphics[width=\textwidth]{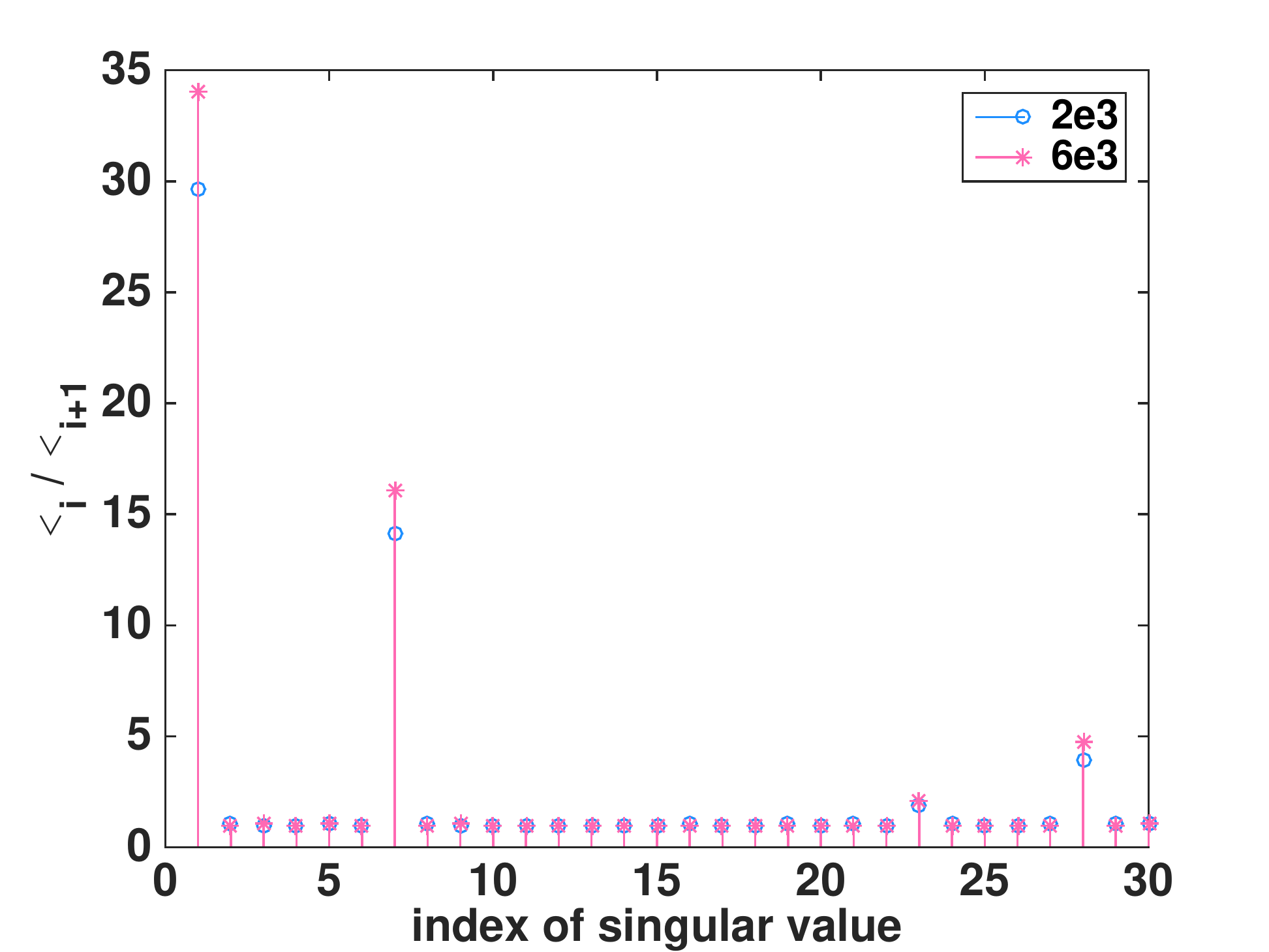}
                \caption{$d = 6$, $\frac{1}{1+\norm{\vecx-\vecy}_2^2/h^2}$}
                \label{fig:sval_ratio_d6_cauchy}
            \end{subfigure}
        \caption{Singular value ratio $\sigma_i / \sigma_{i+1}$ vs.\ index $i$. The singular values are ordered such that
        $\sigma_1 \ge \sigma_2 \ge \cdots \ge \sigma_n$,
        and the legend represents the data size (matrix dimension).
        Subplot (a) and (b) used Gaussian kernel $\exp(-\norm{\vecx-\vecy}_2^2/h^2)$ and subplot (c) and (d) used Cauchy kernel $1 / (1+\norm{\vecx-\vecy}_2^2/h^2)$.}
        \label{fig:sval_ratio}
    \end{figure}

    We study the group cardinality in detail. Consider \Cref{fig:sval_ratio_d3_gaussian}. We consider first the groups separated by significant decays. The indices with ratios above 4 are as follows with the ratio shown in parenthesis,
\[
    1~(17.3),~4~(17.1),~10~(7.3),~20~(4.5)
\]
The indices can be accurately described as the cumulative sum of the number of separate terms in the following Taylor expansion terms from the first-order Fourier term,
\[
    \underbrace{0th}_{\text{1 term}}, ~\underbrace{1st}_{\text{3 terms}}, ~\underbrace{2nd}_{\text{6 terms}},~ \underbrace{3rd}_{\text{10 terms}}
\]
This term arrangement suggests that the polynomial approximation for the first-order Fourier term contributes to the significant gains in accuracy. We note that the higher-order Fourier terms contribute as well, but with fewer accuracy gains.

We consider next the groups separated by less significant decays. The indices with ratios above 2 are
\[
    1~(17.3),~4~(17.1),~7~(2.3),~10~(7.3),~11~(2.3),~17~(3.5),~20~(4.6)
\]
These subtler gains in accuracy may come from the contributions of other higher-order expansion terms. One possible grouping is as follows, with the Fourier order and the Taylor order shown in order in parenthesis,
\[
    \underbrace{(1,0)}_{\text{1 term}}, ~\underbrace{ (2,0),(3,0),(4,0)}_{\text{3 terms}}, ~ \underbrace{ (1,1)}_{\text{3 terms}},~\underbrace{ (2,1)}_{\text{3 terms}},~\underbrace{ (5,0)}_{\text{1 term}},~\underbrace{ (1,2)}_{\text{6 terms}},~\underbrace{ (3,1)}_{\text{3 terms}}
\]
Applying a cumulative sum of the number of these terms yields the above indices.

    Our explanation adopts the idea of the Fourier-Taylor approach instead of the Chebyshev approach. The key reason is that the Fourier approach allows us to group separate terms into finer sets that contribute to subtler error decays. The Chebyshev approach considers $\norm{\vecx-\vecy}^{2l}$ as a unit, which has $\binom{l+d+1}{d+1}$ separate terms; whereas the Fourier approach considers $\left(\vecrho_\vecx^T\vecrho_\vecy\right)^l$ as a unit, which only involves $\binom{l+d-1}{d-1}$ separate terms. 

    \subsection{Practical Guidance}
    The group pattern in the singular values offers insights to many
    phenomena in practice. One example is the threshold \mrank{}s in
    matrix approximations, namely the input \mrank{} has to increase beyond some
    threshold to observe a further decay in the matrix approximation
    error. In practice, our quantification for the group cardinalities can provide candidate \mrank{} inputs for algorithms that take input as a request \mrank{}.

    We examine the effectiveness of our guidance on two popular RBF
    kernel matrices with different low-rank algorithms. We expect significant decays in the reconstruction
    error around \mrank{} $R = \binom{n+d}{d}$.
    For the leverage-score Nystr\"om method, we oversample 30 and 60 columns for $d=6$ and $d=8$, respectively, and report the mean of reconstruction error out of 5 independent runs.
    \Cref{fig:sampling_columns} shows the reconstruction error as a
    function of the approximation \mrank{}. For all the algorithms, a significant decay in error occurs at rank 1, 7, and 28 for $d = 6$, and at rank 1, 9 and 45 for $d = 8$, in perfect agreements with our expectation. Note there exist several subtle perturbations, and they may be caused by the data layouts and contributions from other expansion terms.

    \begin{figure}[htbp]
        \centering
            \begin{subfigure}[b]{0.45\textwidth}
                \includegraphics[width=\textwidth]{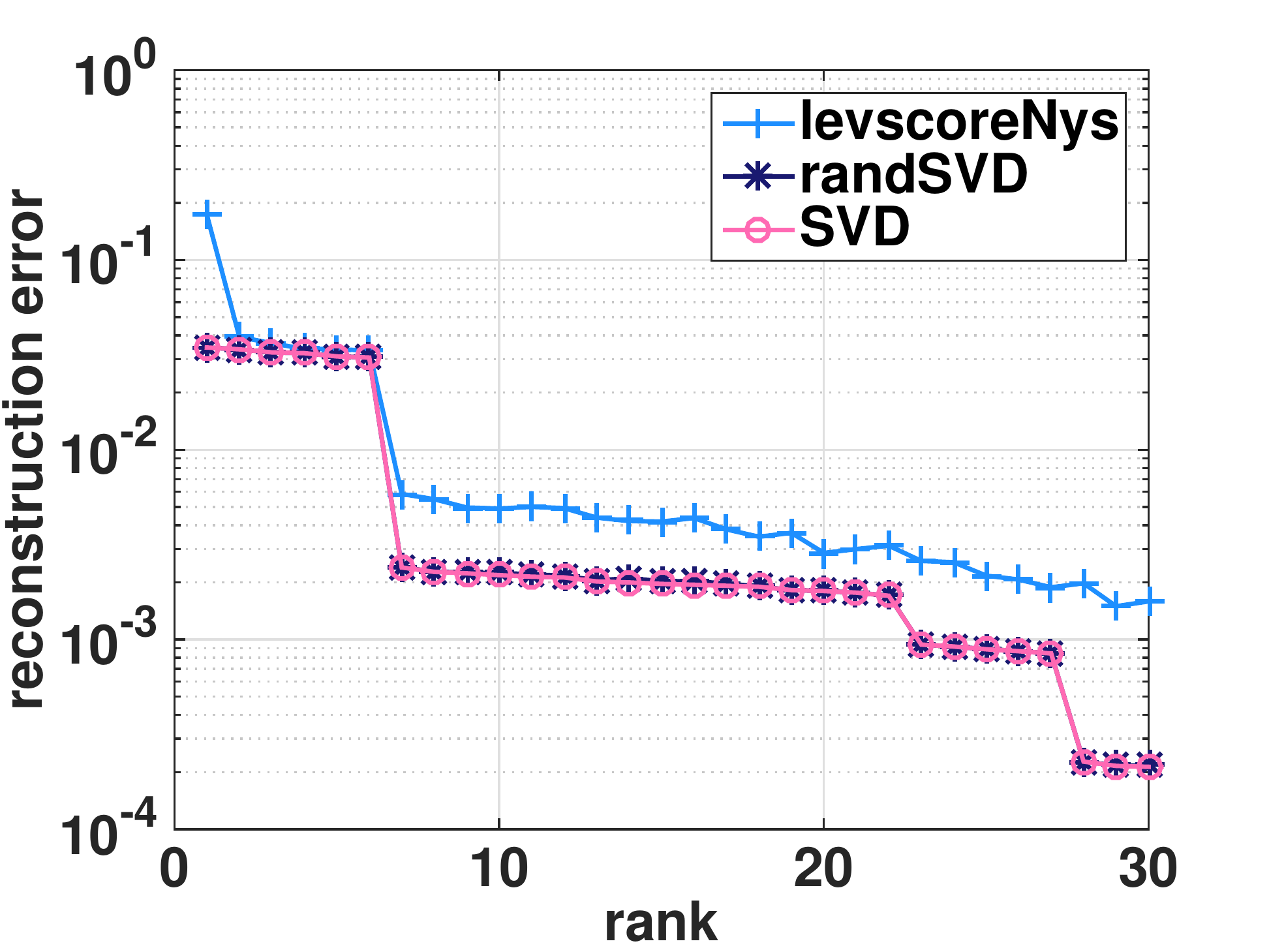}
                \caption{$d = 6$, $\exp(-\norm{\vecx-\vecy}_2^2/h^2)$}
            \end{subfigure}
            \begin{subfigure}[b]{0.45\textwidth}
                \includegraphics[width=\textwidth]{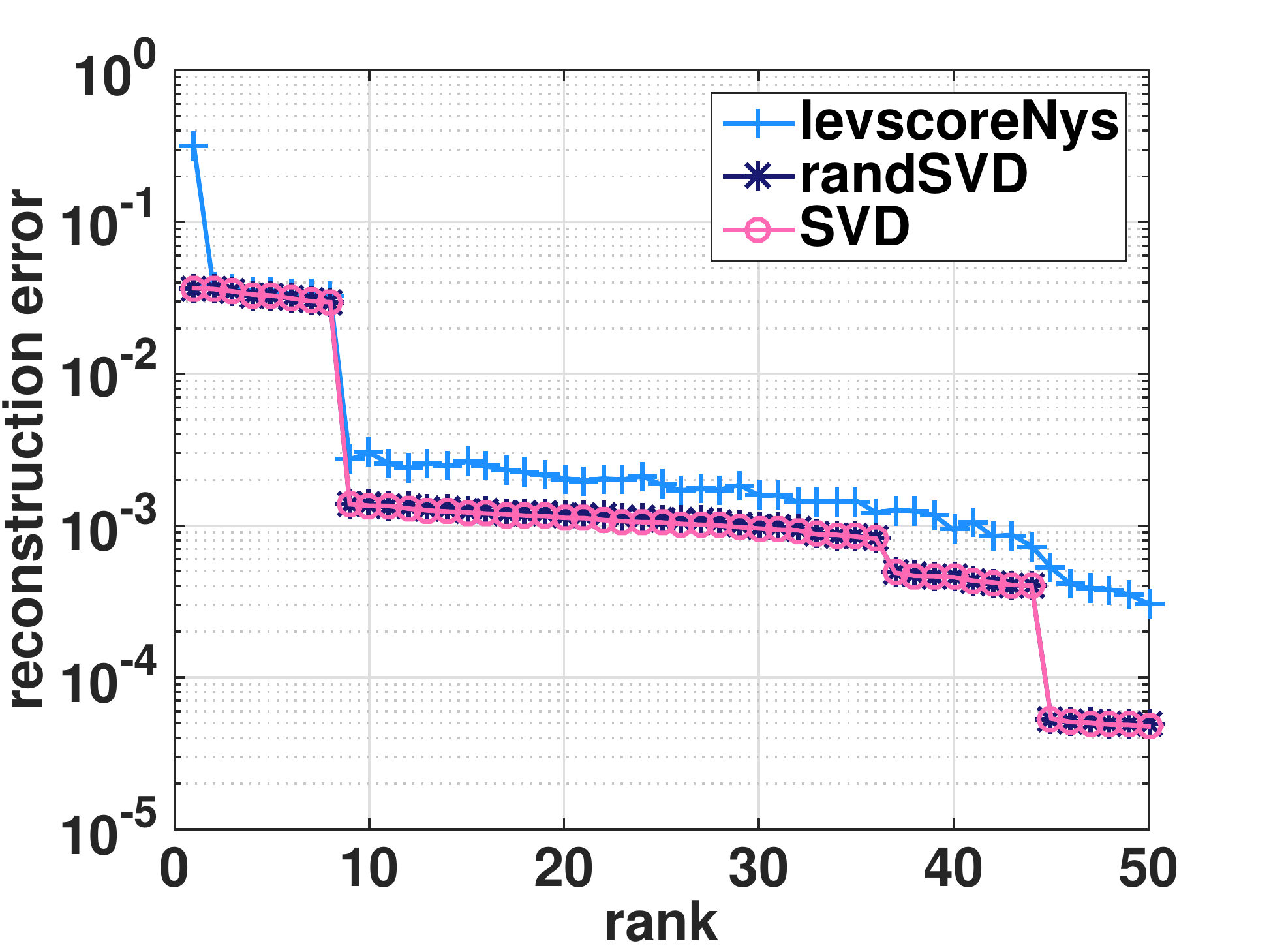}
                \caption{$d = 8$, $\exp(-\norm{\vecx-\vecy}_2^2/h^2)$}
            \end{subfigure}
               \begin{subfigure}[b]{0.45\textwidth}
                   \includegraphics[width=\textwidth]{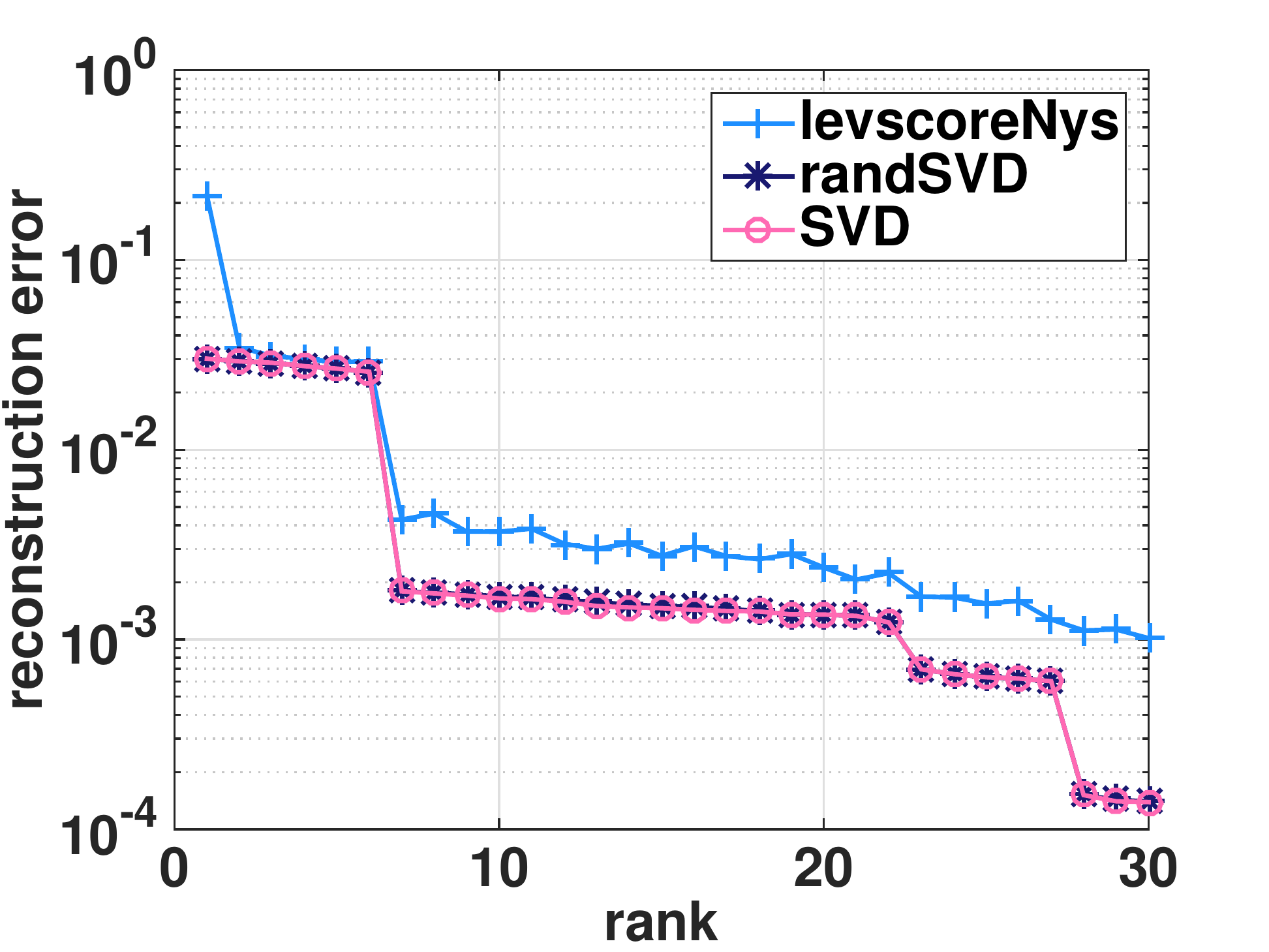}
                   \caption{$d = 6$, $\frac{1}{1+\norm{\vecx-\vecy}_2^2/h^2}$}
               \end{subfigure}
                \begin{subfigure}[b]{0.45\textwidth}
                    \includegraphics[width=\textwidth]{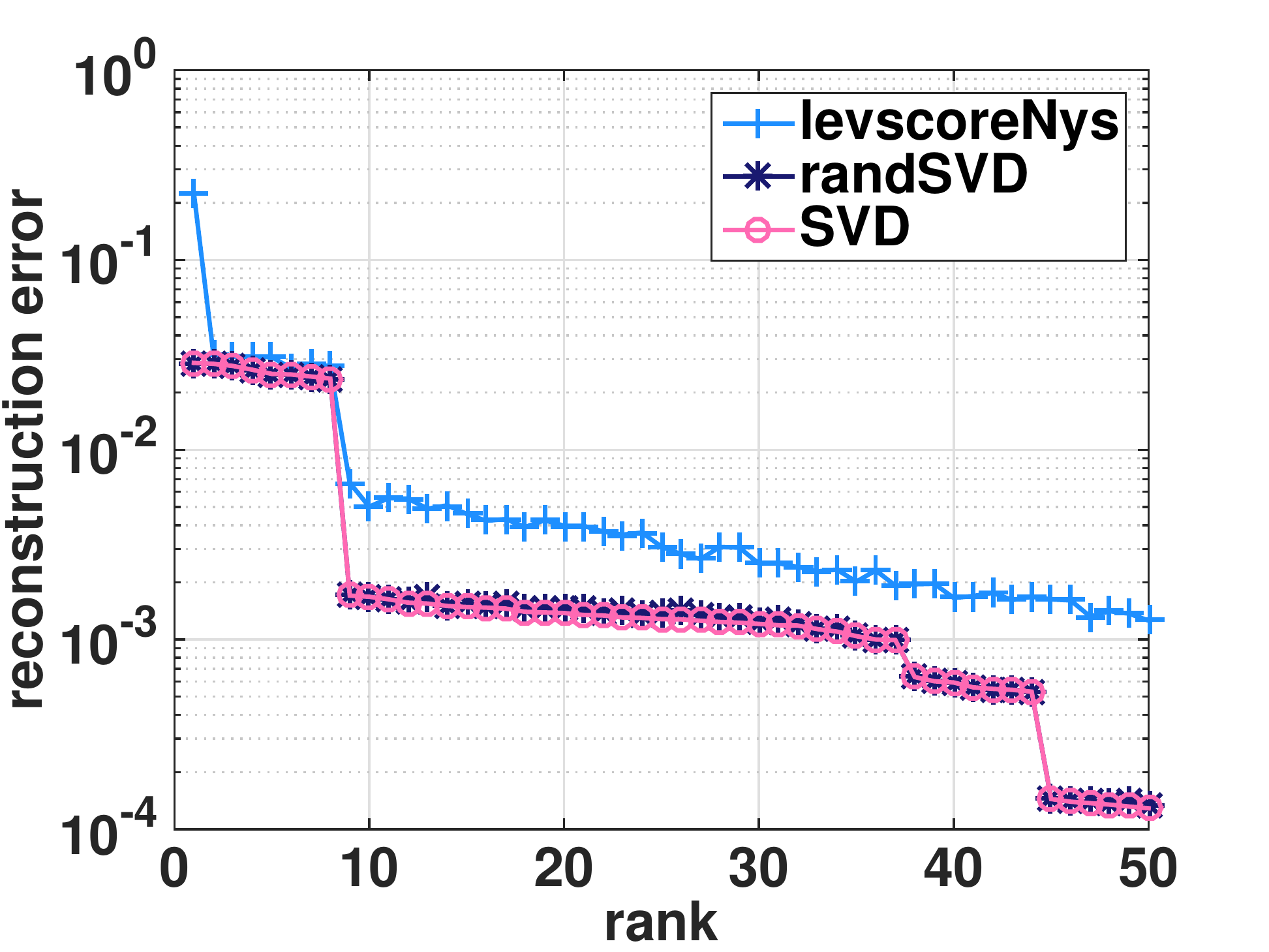}
                    \caption{$d = 8$, $\frac{1}{1+\norm{\vecx-\vecy}_2^2/h^2}$}
                \end{subfigure}
        \caption{Reconstruction error vs.\ approximation rank. The legend represents low-rank algorithms, ``levscoreNys" is the leverage-score Nystr\"om method, ``randSVD" is the randomized SVD with iteration parameter to be 2, and ``SVD" is the exact SVD. The bandwidth parameter $h$ was set to be the maximum pairwise distance. A significant decay in error occurs at rank = $\binom{n+d}{d}$ ($n=1,2,3$) for all experiments.}
        \label{fig:sampling_columns}
    \end{figure}

    \section{Conclusions}

    Motivated by the practical success of low-rank algorithms for RBF
    kernel matrices with high-dimensional datasets,
    we study the \mrank{} of RBF kernel matrices by analyzing its upper bound, that is, the \frank{} of RBF kernels. Specifically, we approximate the RBF kernel by a
    finite sum of separate products, and quantify the 
    upper bounds on the \frank{}s and the $L_\infty$ error for such approximations in their explicit formats. Our three main results are as follows. 

    First, for a fixed precision, the \frank{} of an RBF is a power of data dimension $d$ in the worst case, and the power is related to the precision. The exponential growth for multivariate functions from a simple analysis is absent for RBFs. 

    Second, for a fixed \frank{}, the approximation error will be reduced when the diameters of either the target domain or the source domain decrease.

    Third, we observed group patterns in the magnitude of singular values of RBF kernel matrices. We explained this by our analytic expansion of the kernel function. Specifically, the number of singular values of the same magnitude can be computed by an appropriate grouping of the separate terms in the function's separable form. Very commonly, the cardinality of the $i$-th group is $\binom{i+d-1}{d-1}$.

\newpage
\bibliographystyle{siam}
\bibliography{TBBF2017}
\end{document}